\documentclass[12pt]{amsart}

\def \bs{{\backslash}}

\def \CC{{\mathbb C}}
\def \QQ{{\mathbb Q}}

\def \PP{{\mathbb P}}

\def \Dd{{\mathcal D}}

\def\CH{{\text{CH}}}

\def\del{{\partial}}
\def\ol{\overline}
\def\ul{\underline}
\def\Gy{{\text{Gy}}}
\def\MHS{{\text{MHS}}}

\def\Spec{{\text{Spec}}}

\def\Ext{{\text{Ext}}}

\newtheorem{thm}{Theorem}[section]

\newtheorem{prop}[thm]{Proposition}

\newtheorem{theorem}{Theorem}[section]
\newtheorem{lemma}[theorem]{Lemma}
\newtheorem{corollary}[theorem]{Corollary}
\newtheorem{proposition}[theorem]{Proposition}

\newcommand{\ncom}{\newcommand}
\ncom{\ep}{\epsilon}
\ncom{\rar}{\rightarrow}
\ncom{\thrar}{\twoheadrightarrow}
\ncom{\lrar}{\longrightarrow}
\ncom{\ov}{\overline}
\ncom{\what}{\widehat}

\newcommand{\ignore}[1]{}
\ncom{\m}{\mbox}
\ncom{\sta}{\stackrel}
\ncom{\comx}{{\mathbb C}}
\ncom{\A}{{\mathbb A}}
\ncom{\Z}{{\mathbb Z}}
\ncom{\Q}{{\mathbb Q}}
\ncom{\R}{{\mathbb R}}
\ncom{\G}{{\mathbb G}}
\ncom{\hH}{{\mathbb H}}
\ncom{\al}{\alpha}
\ncom{\p}{{\mathbb P}}
\ncom{\cA}{{\mathcal A}}
\ncom{\cB}{{\mathcal B}}
\ncom{\cD}{{\mathcal D}}
\ncom{\cDB}{{\mathcal D \mathcal B}}
\ncom{\cX}{{\mathcal X}}
\ncom{\cO}{{\mathcal O}}
\ncom{\cW}{{\mathcal W}}
\ncom{\cL}{{\mathcal L}}
\ncom{\cP}{{\mathcal P}}
\ncom{\cH}{{\mathcal H}}
\ncom{\cS}{{\mathcal S}}
\ncom{\cM}{{\mathcal M}}
\ncom{\cC}{{\mathcal C}}
\ncom{\cT}{{\mathcal T}}
\ncom{\cF}{{\mathcal F}}
\ncom{\cN}{{\mathcal N}}
\ncom{\cJ}{{\mathcal J}}
\ncom{\cV}{{\mathcal V}}
\ncom{\cZ}{{\mathcal Z}}
\ncom{\cU}{{\mathcal U}}
\ncom{\cSU}{{\mathcal S \mathcal U}}
\ncom{\cG}{{\mathcal G}}
\ncom{\cQ}{{\mathcal Q}}
\ncom{\cR}{{\mathcal R}}
\ncom{\cY}{{\mathcal Y}}
\ncom{\cE}{{\mathcal E}}
\ncom{\cI}{{\mathcal I}}
\ncom{\mylabel}[1]{{\rm (#1)}\label{#1}}
\ncom{\Hom}{{\textit{Hom}}}
\ncom{\eop}{{\hfill $\Box$}}
\begin{document}
\baselineskip=16pt


\title[The Abel-Jacobi isomorphism for one cycles on Kirwan's log resolution]{The Abel-Jacobi isomorphism for one cycles on Kirwan's log resolution of the moduli space $\cSU_C(2,\cO_C)$}


\author[J. N. Iyer]{Jaya NN  Iyer}

\address{Department of Mathematics and Statistics, University of Hyderabad, Gachibowli, Central University P O, Hyderabad-500046, India}
\address{The Institute of Mathematical Sciences, CIT
Campus, Taramani, Chennai 600113, India}
\email{jniyer@imsc.res.in}

\address{Department of Mathematical and Statistical Sciences, 
University of Alberta, 
632 Central Academic Building, 
Edmonton, AB T6G 2G1}
\email{jdlewis@ualberta.ca}

\footnotetext{Mathematics Classification Number: 53C05, 53C07, 53C29 }

\begin{abstract}
In this paper, we consider the moduli space $\cSU_C(r,\cO_C)$ of rank $r$ semistable vector bundles with trivial determinant  on a smooth projective curve $C$ of genus $g$. When the rank $r=2$, F. Kirwan constructed a smooth log resolution $\ov{X}\rar \cSU_C(2,\cO_C)$. Based on earlier
work of M. Kerr and J. Lewis, Lewis explains in the Appendix the notion of a relative Chow
group (w.r.to the normal crossing divisor), and a subsequent
Abel-Jacobi map on the  relative Chow group  of null-homologous one cycles (tensored
with $\Q$). This map takes values in the  intermediate Jacobian of the compactly supported cohomology of the stable locus. We show that this is an isomorphism and since the intermediate Jacobian is identified with the Jacobian $Jac(C)\otimes \Q$, this can be thought of as a weak-representability result for open smooth varieties. A Hard Lefschetz theorem is also proved for the odd degree bottom weight cohomology of the moduli space $\cSU_C^s(2,\cO_C)$. When the rank $r\geq 2$, we compute the codimension two rational Chow groups of $\cSU_C(r,\cO_C)$.
\end{abstract}
\maketitle


\setcounter{tocdepth}{1}
\tableofcontents

\section{Introduction}

Suppose $C$ is a smooth connected complex projective curve of genus $g$. For a line bundle $L$ on $C$ consider the moduli space $\cSU_C(r, L)$ of semi-stable vector bundles of rank $r$ and of fixed determinant $L$ on $C$. In this paper we assume that either $L:=\cO_C$ or $L:=\cO_C(x)$ for some point $x\in C$. When $L=\cO_C$ the moduli space $\cSU_C(r,\cO_C)$ is a (singular) normal projective variety and when $L=\cO_C(x)$ the moduli space $\cSU_C(r,\cO(x))$ is a smooth projective variety. The space of stable bundles $\cSU_C^s(r,\cO_C)\subset \cSU_C(r,\cO_C)$ forms a smooth quasi-projective variety. Both the moduli spaces have dimension equal to $(r^2-1)(g-1)$.

It is known that these moduli spaces are unirational \cite{SeshadriAst}. This implies that the rational Chow group of zero cycles is trivial, i.e., $\CH_0(\cSU_C(r,L);\Q)\simeq \Q$. More generally, when $X$ is any smooth unirational variety of dimension equal to $n$ then it is well-known that the Hodge groups $H^p(X,\Omega_X^q)=0$ whenever $p=0$. In this case the intermediate Jacobian $IJ^p(X):=\frac{H^{2p-1}(X,\comx)}{F^p+H^{2p-1}(X,\Z)}$ for $p=2,n-1$ is an abelian variety. It is of interest to show the \textit{weak representability} via the Abel-Jacobi map
$$
\CH^p(X)_{\hom} \sta{AJ^p}{\lrar} IJ^p(X)
$$
for $p=2,n-1$ and also determine the abelian variety in terms of the geometry of $X$.
Some examples of Fano threefolds $F$ were shown to have \textit{weakly representable} $\CH_1(F)$ by Bloch and Murre \cite{Bl-Mr} (see also \cite{BlochSrinivas}, for a more general statement of representability for codimension two cycles).

When $X=\cSU_C(2, \cO(x))$, we have the following results;

\begin{eqnarray*}
\CH^1(\cSU_C(r,\cO(x));\Q) & \simeq & \Q   \mbox{ \cite{Ra}}\\
\CH^2(\cSU_C(2,\cO(x));\Q) & \simeq & \CH_0(C;\Q),\,\,g=2    \\
                 & \simeq        & \CH_0(C;\Q) \oplus \Q,\,\,g> 2   \mbox{ \cite{BalajiKingNew}} \\
\CH_1(\cSU_C(2,\cO(x));\Q) & \simeq & \CH_0(C;\Q)  \mbox{ \cite{ChoeHwang}}.
\end{eqnarray*}

Set $n:=(r^2-1)(g-1), \,AJ_1:=AJ^{n-1}$ and suppose $X= \cSU_C(r,\cO_C)$. Since the moduli space $\cSU_C(r,\cO_C)$ is a singular variety, we do not have good Abel-Jacobi maps and perhaps other Chow/cohomology theories should be considered. However, we notice that the open stable locus $U$ is a smooth variety and,
using \cite{Ar-Sa} and Hard Lefschetz, obtain that the compactly supported cohomology $H^{2n-3}_c(U,\Q)$ has a pure Hodge structure. This motivates us to look for an appropriate group of one-cycles, and which admits an Abel-Jacobi map into the intermediate Jacobian $IJ(H^{2n-3}_c(U,\Q))$, that being the import of
the Appendix. 
In particular, the \textit{relative Chow groups} $\ul{\CH}^*(\ov{X},\ov{Y};\Q)$ (w.r.to a smooth compactification) is shown
to admit Abel-Jacobi maps into the intermediate Jacobian of compactly supported cohomologies.

With this in mind, when the rank $r=2$, we consider the log resolution $(\ov{X},\ov{Y})\rar \cSU_C(2,\cO_C)$ studied by F. Kirwan \cite{Ki}, and for which we have an explicit description of the normal crossing divisor $\ov{Y}\subset \ov{X}$. We show that the relative Chow group of one-cycles homologous to zero (up to higher Chow cycles supported on $\bar{Y}$) is isomorphic to $IJ(H^{2n-3}_c(U,\Q))$.   This is similar to the above weak representability results, since we check that this intermediate Jacobian is the Jacobian $Jac(C)\otimes \Q$,  see Corollary \ref{lefhighr}.
  
More precisely, with notations as above, we show:

\begin{theorem}\label{main}
Suppose $g\geq 3$, $r\geq 2$ and fix any point $x\in C$. Then the Abel-Jacobi map on the rational Chow group of one cycles homologous to zero
\begin{equation}
\CH_1(\cSU_C(r,\cO(x));\Q)_{\hom} \sta{AJ_1}{\lrar} Jac(C)\otimes \Q
\end{equation}
is always surjective.
Suppose the rank $r=2$ and $(\ov{X},\ov{Y})\rar \cSU_C(2,\cO_C)$ is Kirwan's log resolution. Then there is an induced Abel-Jacobi map 
\begin{equation*}
\frac{\ul{\CH}^{n-1}(\ov{X},\ov{Y};\Q)_{\hom}}{\ul{\CH}^{n-1}(\ov{Y},1;\Q)^o} \lrar Jac(C)\otimes \Q,
\end{equation*}
which is an isomorphism, whenever $g\geq 4$.
\end{theorem}

We note that the Generalized Hodge Conjecture implies the first assertion. But we use the
correspondence cycle $c_2(U)$ (where $U$ is a universal Poincar\'e bundle \cite{Na-Ra}) between the Chow groups of $0$-cycles on $C$ and the codimension $2$ cycles on $\cSU_C(r,\cO(x))$. Together with the isomorphism of the Lefschetz operator on $H^3(\cSU_C(r,\cO(x)),\Q)$, we conclude the surjectivity of the Abel-Jacobi map on one cycles. In the case of the non compact smooth variety $\cSU_C^s(r,\cO_C)$, a Hecke correspondence was used in \cite{Ar-Sa} for computing the low degree cohomology groups in terms of the cohomology of $\cSU_C(r,\cO(-x))$. We use this correspondence to relate the Chow group of one cycles of these two spaces, for any $r\geq 2$ and $g\geq 4$. 

To prove the second assertion, we first need to find a minimal generating set of one cycles on $\cSU_C(r,\cO_C)$. The Hecke curves \cite{Na-Ra2} are minimal rational curves on the moduli space $\cSU_C(r,\cO_C)$. A variant of a theorem of Kollar \cite[Proposition 3.13.3]{Ko} on Chow generation, is proved by J-M. Hwang and he observed that the Hecke curves generate the rational Chow group of one cycles on $\cSU_C(r,\cO_C)$, see Proposition \ref{generator}, Corollary \ref{Hecke}.  
This gives a surjective map
$$
Jac(C)\otimes \Q \rar \CH_1(\cSU_C(r,\cO_C);\Q)_{\hom}. 
$$

When $r=2$, we notice that the Hecke curves together with some irreducible rational curves lying on the divisor components of $\ov{Y}\subset \ov{X}$ generate the Chow group of one cycles on $\ov{X}$. It suffices to show that $\CH_1(\ov{X};\Q)_{\hom}\simeq Jac(C)\otimes \Q$, see Corollary \ref{motive}.
To obtain our original goal, we need to check that the Gysin map at the level of Chow groups is zero (see \eqref{E27}). This is shown in Lemma \ref{Gysinzero}.

When $r=2$, we investigate further properties of the Chow groups and the cohomology of the moduli space $\cSU_C(2,\cO_C)$, in \S \ref{HL}.
We look at the codimension two cycles on the non-compact smooth variety $\cSU_C^s(r,\cO_C)$ and prove the isomorphism 
$$
\CH^2(\cSU_C^s(r,\cO_C);\Q)_{\hom} \simeq Jac(C)\otimes \Q
$$ 
whenever $g\geq 3$, see Proposition \ref{isoPHS}.

In general, the Hard Lefschetz isomorphism is not true for the rational cohomology of open smooth varieties. In our situation, we show that the Hard Lefschetz isomorphism holds for the bottom weight cohomologies of $\cSU_C^s(2,\cO_C)$ in any odd degree, see Lemma \ref{supportcoh}. This part could be of independent interest also.

\begin{proposition}
Suppose $\cL_s$ is an ample class on the moduli space $\cSU_C^s(2,\cO_C)$. Let $W_iH^i(\cSU_C^s(2,\cO_C))$ denote the bottom weight cohomology group of $\cSU_C^s(2,\cO_C)$, for any $i$. Then the Lefschetz operator $\cL^{i}_s$ gives
an isomorphism
$$
\cL^{i}_s: W_{n-i}H^{n-i}(\cSU_C^s(2,\cO_C),\Q) \sta{\simeq}{\rar} W_{n+i}H^{n+i}(\cSU_C^s(2,\cO_C),\Q)
$$
whenever $n-i$ is odd and $g\geq 3$.
\end{proposition}
This is proved using the desingularisation $\cS$ of $\cSU_C(2,\cO_C)$ constructed by Seshadri (\cite{Se2}) and using the description of the exceptional locus. The exceptional locus has a stratification given by the rank of a conic bundle \cite{Ba}, \cite{Ba-Se}. The computation of the low degree cohomology of $\cS$ in \cite{Ba-Se} used the Thom-Gysin sequences for the stratification. We also use similar long exact seqences of Borel-Moore homologies for the stratification of $\cS$.  A closer analysis reveals that the bottom weight Borel-Moore homologies $W_{-n+i}H_{n-i}$ in odd degree of the exceptional divisor 
are zero. This suffices to conclude the Hard Lefschetz isomorphism on the bottom weight odd degree cohomologies of $\cSU_C^s(2,\cO_C)$. 

In the higher rank $r>2$ case, we pose the question of proving the Hard Lefschetz theorem for appropriate even/odd degree cohomology groups and this could deserve some attention in the future.

 It would be of interest to obtain similar Abel-Jacobi isomorphisms as above for the moduli spaces $\cSU_C(r,L)$, when $L\neq \cO_C$. The proofs given for the surjectivity of the Abel-Jacobi maps using the Hecke correspondence when $L=\cO_C$  would still hold to give the Abel-Jacobi surjectivity
when $L\neq \cO_C$, in most cases. The Chow generation for one cycles given by the Hecke curves seems difficult to prove for general $L$.

{\Small
Acknowledgements: This work was initiated at KIAS, Seoul, on an invitation by Jun-Muk Hwang during Feb 2009.
We are grateful to him for asking this question and for all the discussions we had on the subject and also for extending support and hospitality. In particular he provided the results on Chow generation in \S \ref{last}. We are also grateful to James Lewis for fruitful communications on Abel-Jacobi maps and ideas which led to computations for Kirwan's resolution. He also allowed us to include his results in the Appendix, which motivated some proofs in the main text.  }


\section{Abel-Jacobi surjectivity for one cycles on $\cSU_C(r,\cO(x))$}

Suppose $C$ is a smooth connected projective curve defined over the complex numbers of genus $g$.
Fix a point $x\in C$ and consider the moduli space $\cSU_C(r,\cO(x))$ of stable vector bundles of rank $r$ and fixed determinant $\cO(x)$ on $C$. Atiyah and Bott \cite{At-Bo} have described the generators of the cohomology ring $H^*(\cSU_C(r,\cO(x)), \Q)$ in terms of the characteristic classes of the Poincar\'e bundle. Since we are concerned only with certain cohomologies in low degree, we will recall the generators in low degrees. We also note that the moduli space $\cSU_C(r,\cO(x))$ is isomorphic to $\cSU_C(r,\cO(-x))$ given by $E\mapsto E^*$, the dual of $E$. In the next section we will consider the moduli space $\cSU_C(r,\cO(-x))$ and relate it with the results from this section.

Fix a Poincar\'e bundle $U\rar C\times \cSU_C(r,\cO(x))$ and denote the $i$-th Chern class of $U$ by $c_i(U)\in H^{2i}(C\times \cSU_C(r,\cO(x)),\Q)$. Denote the two projections by 
$$
p_1: C\times \cSU_C(r,\cO(x))\rar C,\,p_2:C\times \cSU_C(r,\cO(x))\rar \cSU_C(r,\cO(x)).
$$
 The cycle $c_i(U)$ acts as a correspondence between the cohomology of $C$ and the cohomology of $\cSU_C(r,\cO(x))$. More precisely, there are homomorphisms:
\begin{eqnarray*}
H^k(C,\Q) \sta{p_1^*}{\lrar} H^k(C\times \cSU_C(r,\cO(x)),\Q) \sta{\cup c_i(U)}{\lrar}H^{k+2i}(C\times \cSU_C(r,\cO(x)),\Q) \\
 H^{k+2i}(C\times \cSU_C(r,\cO(x)),\Q)\sta{{p_2}_*}{\lrar} H^{k+2i-2}(\cSU_C(r,\cO(x)),\Q).
\end{eqnarray*}
The composition $\Gamma_{c_i(U)}:= {p_2}_*\circ \cup c_i(U)\circ p_1^*$ is called the correspondence defined by the cycle $c_i(U)$.

\begin{theorem}\label{Narasimhan}
The correspondence
$$
\Gamma_{c_2(U)}: H^1(C,\Q)(-1) \lrar H^3(\cSU_C(r,\cO(x)),\Q)
$$
is an isomorphism, for $r\geq 2, g\geq 3$.
\end{theorem}
\begin{proof}
See \cite[Theorem 3]{Na-Ra} together with \cite[Section 4 and Lemma 2.1]{Ra}. The isomorphism is actually with integral coefficients.
\end{proof}

Fix an ample line bundle $\cO(1)$ on the moduli space $\cSU_C(r,\cO(x))$. Denote its class $H:=c_1(\cO(1)\in H^2(\cSU_C(r,\cO(x),\Q)$.
Set $n:=\m{dim} \cSU_C(r,\cO(x))$.
\begin{corollary}\label{isoHS}
The composition
$$
 H^1(C,\Q) \sta{ \Gamma_{c_2(U)}}{  \lrar} H^3(\cSU_C(r,\cO(x)),\Q) \sta{\cup H^{n-3}}{\lrar} H^{2n-3}(\cSU_C(r,\cO(x)),\Q)(n-2)
$$
is an isomorphism of pure Hodge structures.
\end{corollary}
\begin{proof}
This follows from Theorem \ref{Narasimhan} and the Hard Lefschetz theorem.
\end{proof}

\begin{corollary}\label{isoIJ}
There is an isomorphism of the intermediate Jacobians
$$
Jac(C)\otimes \Q \lrar IJ_1:= \frac{H^{2n-3}(\cSU_C(r,\cO(x)),\comx)}{F^{n-1}+H^{2n-3}(\cSU_C(r,\cO(x)),\Q)}
$$
induced by the composed morphism $\cup H^{n-3}\circ \Gamma_{c_2(U)}$.
\end{corollary}

We now want to show that the isomorphisms on the intermediate Jacobians are compatible with intersections and correspondences on the rational Chow groups.
More precisely, consider the classes $c_2(U)\in \CH^2(\cSU_C(r,\cO(x)))$ and $H\in \CH^1(\cSU_C(r,\cO(x)))$.
Set $\Gamma_{c_i(U)}^{CH}:= {p_2}_*\circ \cap c_i(U)\circ p_1^*$, where
\begin{eqnarray*}
\CH^1(C;\Q) \sta{p_1^*}{\lrar} \CH^1(C\times \cSU_C(r,\cO(x));\Q) \sta{\cap c_2(U)}{\lrar}\CH^{3}(C\times \cSU_C(r,\cO(x));\Q) \\
 \CH^{3}(C\times \cSU_C(r,\cO(x));\Q) \sta{{p_2}_*}{\lrar} \CH^{2}(\cSU_C(r,\cO(x));\Q).
\end{eqnarray*}

We consider the correspondence
$$
\Gamma_{c_2(U)}^{\CH}: \CH^1(C;\Q) \lrar \CH^2(\cSU_C(r,\cO(x));\Q).
$$

This restricts to a correspondence on the subgroup of cycles homologous to zero:
$$
\Gamma_{c_2(U)}^{\CH}: \CH^1(C;\Q)_{\hom} \lrar \CH^2(\cSU_C(r,\cO(x));\Q)_{\hom}.
$$
Furthermore, there is a composition of morphisms:
\begin{equation}\label{psi}
\CH^1(C;\Q)_{\hom} \sta{\Gamma_{c_2(U)}^{\CH}}{\lrar} \CH^2(\cSU_C(r,\cO(x));\Q)_{\hom} \sta{\cap H^{n-3}}{\lrar} \CH_1( \cSU_C(r,\cO(x));\Q)_{\hom}.
\end{equation}
Set $\psi:= \cap H^{n-3}\circ \Gamma_{c_2(U)}^{\CH}$.

We recall the exact sequence which relates the intermediate Jacobian with the Deligne cohomology group \cite{Es-Vi}:
\begin{equation}\label{Deligne}
0\lrar IJ_1\rar H^{2n-2}_\cD(\cSU_C(r,\cO(x)),\Z(n-1))\lrar Hg^{n-1}\big(\cSU_C(r,\cO(x))\big)\lrar 0.
\end{equation}
Here $Hg^{n-1}:= \epsilon^{-1}\big(H^{n-1,n-1}(\cSU_C(r,\cO(x))\big)$, where
\[
\epsilon : H^{2n-2}(\cSU_C(r,\cO(x)),\Z)\rar H^{2n-2}(\cSU_C(r,\cO(x)),\CC),
\]
is the natural map.

Recall the Abel-Jacobi maps:
\begin{equation*}
\CH^1(C)_{\hom}\sta{AJ_C}{\lrar} Jac(C),
\end{equation*}
\begin{equation*}
\CH_1(\cSU_C(r,\cO(x)))_{\hom}\sta{AJ_1}{\lrar} IJ_1.
\end{equation*}

\begin{lemma}\label{commute}
The map $\psi$ defined in \eqref{psi} is compatible with the Abel-Jacobi maps $AJ_C$ and $AJ_1$. In other words, the following diagram commutes:
\begin{eqnarray*}
\CH^1(C;\Q)_{\hom} & \sta{\psi}{\lrar} & \CH_1(\cSU_C(r,\cO(x));\Q)_{\hom} \\
\downarrow AJ_C &                           & \downarrow AJ_1 \\
Jac(C)\otimes \Q\,\,\,\,\,\,\,\,\,\,\, & \sta{\simeq}{\lrar } & IJ_1\otimes \Q.
\end{eqnarray*}
\end{lemma}
\begin{proof}
There are cycle class maps
$$
\CH_1(\cSU_C(r,\cO(x)))\rar H^{2n-2}_\cD(\cSU_C(r,\cO(x)),\Z(n-1))
$$
$$
\CH^1(C) \rar H^2_\cD(C,\Z(1))
$$
which induce the Abel -Jacobi map on the subgroups of cycles which are homologous to zero (see \cite{Es-Vi}). In other words, using the exact sequence \eqref{Deligne}, we note that the Deligne cycle class map on the subgroup of cycles homologous to zero factors via
the intermediate Jacobian and this map is the same as the Abel-Jacobi map.  Furthermore, the cycle class map into the Deligne cohomology is compatible with correspondences and intersection products on the Chow groups.
This implies that the above diagram in the statement of the lemma is commutative. The isomorphism on the last row of the commutative diagram is given by Corollary \ref{isoIJ}.
\end{proof}

\begin{corollary}\label{AJsurj}
The Abel-Jacobi map
$$
AJ_1:  \CH_1(\cSU_C(r,\cO(x));\Q)_{\hom} \lrar IJ_1\otimes \Q
$$
is surjective and is a splitting, i.e., the inverse $AJ_1^{-1}$ is well-defined and is injective.
\end{corollary}
\begin{proof}
Use the isomorphism in the last row of the commutative diagram of Lemma \ref{commute} to obtain the surjectivity of $AJ_1$ and a splitting.
\end{proof}

We know by Corollary \ref{AJsurj} that the Abel-Jacobi map $AJ_1$ is surjective and a splitting. To show that the Abel-Jacobi map is actually an isomorphism, it suffices to show that the the one-cycles on the moduli space are generated by cycles parametrised by the Jacobian $Jac(C)$.  In other words there should be a surjective map
\begin{equation}\label{assumption}
Jac(C)\otimes \Q \rar \CH_1(\cSU_C(r,\cO(x)))_{\hom}\otimes \Q.
\end{equation}

We will see in \S \ref{last} that the assumption \eqref{assumption} is fulfilled for $\cSU_C(r,\cO_C)$, whenever $r\geq 2$ and $g\geq 4$. This will help us to conclude the desired Abel-Jacobi isomorphism.

Since the moduli space of stable bundles $\cSU_C^s(r,\cO_C)\subset  \cSU_C(r,\cO_C)$ is an open subvariety, we need to first explain how the Abel-Jacobi map
is defined for non-compact varieties. In our situation,it requires us to introduce the relative Chow group. We will see that Abel-Jacobi maps can be defined on the cycles homologous to zero in these groups. 
Before studying the Abel-Jacobi maps, we first try to compute the rational Chow groups and
cohomology groups in small codimensions and obtain a Hard Lefschetz Theorem. These results
will enable us to obtain appropriate target groups for the Abel-Jacobi maps.
 
\section{Hard Lefschetz isomorphism for the bottom weight odd degree rational cohomology of $\cSU_C^s(2,\cO_C)$}\label{HL}

In this section, we investigate the Chow/cohomology properties of the moduli space $\cSU_C(2,\cO_C)$.
The main result is a proof of the Hard Lefschetz isomorphism for the bottom weight odd degree rational cohomology of the moduli space $\cSU_C(2,\cO_C)$. Along the way, we also obtain an isomorphism of the codimension two Chow group of $\cSU_C^s(2,\cO_C)$ with the Jacobian $Jac(C)$, with $\Q$-coefficients.
The Lefschetz operator acts compatibly on the Chow groups and on the cohomology groups. This enables us to give a proof of the surjectivity of the Abel-Jacobi map on one cycles, when  $g\geq 3$.
In the next section, we investigate the Abel-Jacobi map surjectivity for the higher rank $r\geq 2$ case, under the assumption $g\geq 4$.

The main tool is the \textit{Hecke correspondence} which relates the Chow groups and the cohomology groups of the moduli space $\cSU_C(2,\cO(-x))$ with those of the moduli space  $\cSU_C(2,\cO_C)$. 
We recall the Hecke correspondence used by Arapura-Sastry \cite{Ar-Sa} to study the cohomology of the moduli space $\cSU_C^s(r, \cO_C)$ for any $r\geq 2$, since this correspondence will also be used in the next section.

As in the previous section, fix a point $x\in C$. Consider the moduli space $\cSU_C(r,\cO(-x))$ which is a smooth projective variety.
Then there exists a Poincar\'e bundle $\cP\lrar C\times \cSU_C(r,\cO(-x))$. Then $\cP$ restricted to $C\times [E]$ corresponds to a stable bundle $E$ on $C$ and $\cP$ is unique upto pullback of line bundles from $\cSU_C(r,\cO(-x))$. Let $\cP_x$ denote the restriction of $\cP$ to
$\{x\}\times \cSU_C(r,\cO(-x))$. It is known that $\cP_x$ and $\cP$ are stable bundles with respect to any polarization on $\cSU_C(r,\cO(-x))$ and
$C\times \cSU_C(r,\cO(-x))$, respectively (see \cite{Na-Ra}).

Denote the projectivisation $\p:=\p(\cP_x)$ and
 $\pi:\p\lrar \cSU_C(r,\cO(-x))$ be the projection. 

There is a \textit{Hecke diagram} relating the moduli spaces $\cSU_C(r,\cO(-x))$ and $\cSU_C(r,\cO_C)$:
\begin{eqnarray*}
\p\,\,\,\,\,\,\,\,\, & \sta{f}{\lrar} & \cSU_C(r,\cO_C) \\
\downarrow \pi \,\,\,\,&     & \\
\cSU_C(r,\cO(-x)).
\end{eqnarray*}
It is shown in \cite[\S 5]{Ar-Sa} that the space $\p$ is a fine moduli space of quasi-parabolic bundles and parametrises quasi-parabolic structures $V\lrar \cO_Z$ whose kernel is semi-stable. 
Furthermore, there is an open set $U:=f^{-1}\cSU_C^s(r,\cO_C) \subset \p$ such that $f^{-1}V\simeq \p(V_x^*)$ for $V\in \cSU_C^s(r,\cO_C)$. In particular $f$ restricts to a projection
\begin{equation}
f_U:U\lrar \cSU_C^s(r,\cO_C)
\end{equation}
which is a $\p^{r-1}$-bundle, see \cite[Remark 5.0.2]{Ar-Sa}.
The diagram
\begin{equation}\label{excision}
\cSU_C(r,\cO(-x))\sta{\pi}{\leftarrow} \p \supset U\sta{f_U}{\rightarrow} \cSU_C^s(r,\cO_C)
\end{equation}
together with Hodge theory, projective bundle formulas and codimension estimates enabled Arapura-Sastry to compare the cohomologies of the two moduli spaces in \eqref{excision} at least in low degrees.
For our purpose, we recall the codimension estimate \cite[p.17]{Ar-Sa};
\begin{equation}\label{codim}
\m{codim}(\p-U)\geq 3
\end{equation}
whenever $g\geq 3$. 
Altogether, the cohomology $H^3(\cSU_C^s(r,\cO_C),\Q)$ has a pure Hodge structure of weight $3$ and there is an isomorphism of pure Hodge structures  \cite[Theorem 8.3.1]{Ar-Sa}:
\begin{equation}\label{HS}
H^3(\cSU_C(r,\cO(-x)),\Q) \simeq H^3(\cSU_C^s(r,\cO_C),\Q).
\end{equation}

Together with the isomorphism in Theorem \ref{Narasimhan}, there is an isomorphism of Hodge structures:
\begin{equation}\label{ar-sa}
H^1(C,\Q)(-1)\simeq H^3(\cSU_C^s(r,\cO_C),\Q).
\end{equation}

We would like to extend this isomorphism to the cohomology $H^{2n-3}(\cSU_C^s(r,\cO_C)
,\Q)$ as an isomorphism of Hodge structures, via the Lefschetz operator. Since $\cSU_C^s(r,\cO_C)$ is a non-compact smooth variety, the Hard Lefschetz theorem is not immediate. In general, this theorem does not hold for non-compact smooth varieties. Hence, we investigate the action of the Lefschetz operator on suitable sub-structures of the rational cohomology group, where the isomorphism may hold. 

For this purpose, we look at the resolution $\cS$ of $\cSU_C(r,\cO_C)$ constructed by Seshadri \cite{Se2} together with an understanding of the exceptional loci of the resolution
\begin{equation}
g:\cS\lrar \cSU_C(r,\cO_C).
\end{equation}
This map restricts to an isomorphism $g^{-1}\cSU_C^s(r,\cO_C)\simeq \cSU_C^s(r,\cO_C)$.
The variety $\cS$ is constructed as a moduli space of semi-stable vector bundles of rank $r^2$ and trivial determinant whose endomorphism algebra is a specialisation of the matrix algebra. In general, the moduli space $\cS$ is a normal projective variety and it is proved to be a smooth variety only when the rank $r=2$. Hence we assume $r=2$ in the further discussion. See also other resolutions by Narasimhan-Ramanan \cite{Na-Ra2} and Kirwan \cite{Ki}.

\subsection{Stratification of $\cS$ when $r=2$}\label{Seshadri's}

We will recall the description of the exceptional loci of Seshadri's desingularisation given in \cite[\S 3]{Ba}, \cite{Ba-Se}.
Recall that the singular locus of $\cSU_C(2,\cO_C)$ is parametrised by semi-stable bundles of the type $L\oplus L^{-1}$ for
$L\in Jac(C)$. The inverse map $i$ on $Jac(C)$ is given by $L\mapsto L^{-1}$. In other words, the Kummer variety $K(C):= \frac{Jac(C)}{<i>}$ is precisely the singular locus. Denote the image of the set of $2^{2g}$ fixed points by $K(C)_{fix}\subset K(C)$.
There is a stratification
\begin{equation}
\cSU_C(2,\cO_C)= \cSU_C^s(2,\cO_C) \sqcup (K(C)-K(C)_{fix})\sqcup K(C)_{fix}.
\end{equation}

The desingularisation $\cS$ is stratified by the rank of a natural conic bundle on $\cS$ \cite[\S 3]{Ba} and there is a filtration by closed subvarieties
\begin{equation}
\cS=\cS_0\supset \cS_1\supset \cS_2\supset \cS_3
\end{equation}
such that $\cS-\cS_1= g^{-1}(\cSU_C^s(2,\cO_C))$ and $\cS_{i+1}$ is the singular locus of $\cS_i$.

The strata are described by the following:
\begin{proposition}\label{BalajiSeshadri}
1) The image $g(\cS_1-\cS_2)$ is precisely the middle stratum. In fact $\cS_1-\cS_2$ is a
$\p^{g-2}\times \p^{g-2}$ bundle over $K(C)-K(C)_{fix}$.

2) The image of $\cS_2$ is precisely the deepest stratum $K(C)_{fix}$ and $\cS_2-\cS_3$ is the disjoint union of $2^{2g}$ copies of a vector bundle of rank $g-2$ over the Grassmanian $G(2,g)$. The stratum $\cS_3$ is the disjoint union of $2^{2g}$ copies of the Grassmanian $G(3,g)$.
\end{proposition}
\begin{proof}
See \cite[section 2]{Ba-Se}.
\end{proof}

We now note that the strata $\cS_1,\cS_2$ in the stratification $\cS\supset \cS_1\supset \cS_2 \supset \cS_3$ are singular varieties. Moreover $\cS_2\subset \cS_1$ and $\cS_3\subset \cS_2$ are the singular loci respectively.
Hence for our purpose, we look at the Borel-Moore homology theory (for example see \cite[Chapter V, \S 6]{PetersSteenbrink} for properties and the notion of weights) in the below discussion. In the following the Borel-Moore homology is denoted by $H_i(M,\Q)$. This is the usual homology, when $M$ is a complete variety.

Before proving our main result, we need the projective bundle formula of mixed Hodge structures:

\begin{lemma}\label{BMweights}
Suppose $p:M\rar N$ is a $d$-fold fibre product of $\p^{r-1}$-bundles (which need not be locally trivial in the Zariski topology) over a smooth quasi-projective variety $N$.
Then, we have an equality of the mixed Hodge structures;
$$
W_iH^i(M,\Q)=\oplus_{j\geq 0}W_{i-2j}H^{i-2j}(N,\Q)\otimes H^{2j}((\p^{r-1})^{d},\Q).
$$
This gives the equality of the Borel-Moore homologies of $M$ and $N$;
$$
W_{-i}H_i(M,\Q)=\oplus_{j\geq 0}W_{2j-i}H_{i-2j}(N,\Q)\otimes H_{2j}((\p^{r-1})^{d},\Q).
$$
\end{lemma}
\begin{proof}
The proof is basically given in \cite[Proposition 6.3.1]{Ar-Sa}. Their proof is stated in terms of the 'full' cohomology $H^i(N,\Q)$ and $H^i(M,\Q)$, whereas for our purpose we restrict their formula to the bottom weight cohomology of $M$, which gives the formula as stated. Here we use the fact that the product of projective spaces have only pure Hodge structures and that the category of $\Q$-mixed Hodge structures is semi-simple. In fact the same proof also holds for the compactly supported cohomologies $H^i_c(M,\Q)$ and $H^i_c(N,\Q)$ and we can restrict it on the  
weight $i$-piece, to get the formula;
$$
W_iH^i_c(M,\Q)=\oplus_{j\geq 0}W_{i-2j}H^{i-2j}_c(N,\Q)\otimes H^{2j}((\p^{r-1})^{d},\Q).
$$

The second statement follows from the definition (see \cite[Definition-Lemma 6.11]{PetersSteenbrink}) of weights on the Borel-Moore homology of a smooth quasi projective variety $R$ of dimension $d$ using the isomorphism of MHS:
$$
H_k(R,\Q)\simeq H^{2d-k}_c(R,\Q(d)).
$$

\end{proof}

We also recall the following standard facts for the convenience of the reader. The below facts are essentially collected from U. Jannsen's book \cite[\S 6]{Ja} on twisted Poincar\'e duality theories. Perhaps there are other theories which one could also use, but to prove Corollary \ref{oddzero} and 
Lemma \ref{support}, Jannsen's theory seems to suffice.
 
\begin{lemma}\label{mhsfact}
Let $X$ be a complete  variety of pure dimension $n$ and $T\subset X$ be a closed subvariety. 

\noindent i a)  Suppose $X$ is a smooth variety. Then for all $k$ we have a nonsingular pairing of mixed Hodge structures
$$
H^k_T(X) \otimes H^{2n-k}(T) \rar H^{2n}_T(X)\rar \Q(-n).
$$
Here $H^k_T(X)$ denotes the cohomology supported on $T$. 

\noindent i b) There is a long exact sequence
$$
\rar H^k_T(X) \rar H^k(X) \rar H^k(X-T)\rar.
$$
The group $H^k_T(X)$ has only weights $\geq k$ and the image of this group in $H^k(X)$ has weight $k$. In other words, the image is identified with the image of the group $W_kH^k_T(X)$.

\noindent i c) (Poincar\'e duality) There is an isomorphism
$$
H^{2n-k}_T(X)\simeq H_k(T)(-n)
$$ 
of $\Q$-mixed Hodge structures.

ii) Suppose $X$ is a singular and complete variety. Let $T\subset X$ be a closed subvariety. Then there exist dual long exact sequences of cohomologies
$$
\rar H^k_c(X-T) \rar H^k(X) \rar H^k(T)\rar
$$
and
$$
\rar H_k(T) \rar H_k(X) \rar H_k(X-T)\rar.
$$
\end{lemma}
\begin{proof}
i a) See the discussion above \cite[Corollary 6.14]{PetersSteenbrink}.
The second statement i b) follows from  \cite[Corollary 6.14]{PetersSteenbrink} and noting that $H^k(X)$ has a pure Hodge structure of weight $k$.
The third statement i c)follows from \cite[p.82, h) and p.92, Example 6.9]{Ja}.
The statement ii) follows from \cite[p.81,\S 6, f)]{Ja}.
\end{proof}

\begin{corollary}\label{oddzero}
The weight $(-i)$ piece  $W_{-i}H_i(\cS_1,\Q )$ of the homology $H_i(\cS_1,\Q )$ of $\cS_1$ is zero, whenever $i$ is odd.
\end{corollary}
\begin{proof}
We firstly note that the involution $i$ acts as $(-1)$ on the odd homology (respectively $(+1)$ on the even homology) of the Jacobian $J(C)$. The $i$-fixed homology classes precisely correspond to the homology classes on the quotient variety $K(C)$. This means that the homology of the Kummer variety $K(C)$ is zero in odd degrees. Look at the long exact Borel--Moore homology sequence for the triple $(K(C), K(C)-K(C)_{fix}, K(C)_{fix})$ (see Lemma \ref{mhsfact} ii)):
$$
 \rar H_i(K(C)_{fix},\Q)\rar H_i(K(C),\Q)\rar H_i(K(C)-K(C)_{fix},\Q)\rar.
$$
Since $K(C)_{fix}$ is a finite set of points, we conclude that the homology of the open variety $K(C)-K(C)_{fix}$ is also zero in odd degrees, except in degree $i=1$. In this case, we obtain an injectivity $ H_1(K(C)-K(C)_{fix},\Q)\hookrightarrow H_0(K(C)_{fix},\Q)$.
Since $H_0(K(C)_{fix},\Q)$ has weight $0$, we conclude that $H_1(K(C)-K(C)_{fix},\Q)$ also has only weight $0$.
In particular, we have the vanishing $W_{-1}H_1(K(C)-K(C)_{fix},\Q)=0$.
 
Now we look at the homology of the triple $(\cS_2, \cS_2-\cS_3,\cS_3)$:
the long exact Borel--Moore homology sequence for this triple is
$$
 \rar H_i(\cS_3,\Q)\rar H_i(\cS_2,\Q)\rar H_i(\cS_2-\cS_3,\Q)\rar.
$$

Note that the Grassmanian and vector bundles over Grassmanians have only algebraic homology, i.e., have homology only in even degrees.
By  Proposition \ref{BalajiSeshadri} 2), we know that $\cS_2-\cS_3$ and $\cS_3$ are made of such objects. Hence  we conclude that $\cS_2$ has all of its odd degree cohomology  zero.

Now look at the long exact homology sequence for the triple $(\cS_1, \cS_1-\cS_2, \cS_2)$:
here we note that the homology of the projective bundles or more generally flag varieties over a variety are generated by the homology of the base variety and the standard homology classes of powers of $\cO(1)$ on these bundles. In particular these standard classes contribute only in the even homology of the total space. Since we have noticed above that $K(C)-K(C)_{fix}$ has vanishing odd degree homology in degrees $>1$, and in degree one case  we have $W_{-1}H_1( K(C)-K(C)_{fix},\Q)=0$, it follows that $\cS_1-\cS_2$ has  in odd degree $i$,  the vanishing of the bottom weight cohomology $W_{-i}H_i(\cS_1-\cS_2,\Q)=0$, by Lemma \ref{BMweights}. We use the long exact homology sequence for the triple $(\cS_1, \cS_1-\cS_2, \cS_2)$ and since $\cS_2$ has vanishing odd degree homology we conclude, for an odd degree $i$, the vanishing of the bottom weight cohomology $W_{-i}H_i(\cS_1,\Q)=0$.
\end{proof}

\subsection{Hard Lefschetz isomorphism for $W_{odd}H^{odd}(\cSU_C^s(2,\cO_C),\Q)$}

Fix an ample class $\cL$ on the moduli space $\cS$ which restricts to an ample class $\cL_s$ on $\cSU_C^s(2,\cO_C)$. Let $n:=\m{dim }\cS$.
Now we look at the long exact cohomology sequence for the triple $(\cS,\cSU_C^s(2,\cO_C), \cS_1)$ which is compatible with the Lefschetz operators:
\begin{eqnarray*}
\rar H^{n-i}_{\cS_1}(\cS,\Q) &\sta{g}{ \rar} H^{n-i}(\cS,\Q)\sta{h}{\rar} & H^{n-i}(\cSU_C^s(r,\cO_C),\Q)\rar \\
                         & \downarrow \cup \cL^i & \downarrow \cup \cL^i_s  \\
 \rar H^{n+i}_{\cS_1}(\cS,\Q)& \sta{g'}{\rar} H^{n+i}(\cS,\Q)\sta{h'}{\rar} & H^{n+i}(\cSU_C^s(r,\cO_C),\Q)\rar.
\end{eqnarray*}
Here $H^{n-i}_{\cS_1}(\cS,\Q)$ and $H^{n+i}_{\cS_1}(\cS,\Q)$ denotes the cohomology supported on $\cS_1$.

\begin{lemma}\label{support}
The image of the cohomology $H^{n-i}_{\cS_1}(\cS,\Q)$ (respectively $H^{n+i}_{\cS_1}(\cS,\Q)$)  under $g$ (respectively $g'$) is zero, whenever $n-i$ is odd. In particular the maps $h,h'$ in the above long exact sequences are injective whenever $n-i$ is odd. 
\end{lemma}
\begin{proof}
Since $\cS_1\subset \cS$ is a divisor, we can identify the group $H^{n-i}_{\cS_1}(\cS,\Q)$ with the Borel--Moore homology $H_{n+i}(\cS_1,\Q(n))$, which is actually the group $H_{n+i}(\cS_1,\Q)(-n)$, see  Lemma \ref{mhsfact} i c).  This group has weights $2(n)-(n+i) = n-i$ and higher, and the image of $H_{n+i}(\cS_1,\Q(n))$ coincides with the image of
$W_{n-i}(H_{n+i}(\cS_1,\Q)(-n))$, see Lemma \ref{mhsfact} i b).

Now we notice that since $\Q(-n)$ has weight $(2n)$, the group $W_{n-i}(H_{n+i}(\cS_1,\Q)(-n))$ is identified with the group $(W_{-n-i}H_{n+i}(\cS_1,\Q))(-n)$.
Since $n-i,n+i$ are odd, by Corollary \ref{oddzero}, we know that  $W_{-n-i}H_{n+i}(\cS_1,\Q)=0$. Hence we conclude that the image of the cohomology $H^{n-i}_{\cS_1}(\cS,\Q)$ under $g$ is zero. A similar computation shows that the group $H^{n+i}_{\cS_1}(\cS,\Q)$ is identified with the group $W_{-n+i}H_{n-i}(\cS_1,\Q)(-n)$ and hence its image under $g'$ is also zero.
\end{proof}

Since the images of $h$ and $h'$ correspond to the bottom weight  cohomology of $\cSU_C^s(2,\cO_C)$ (see \cite[p.39, Corollaire 3.2.17]{Deligne}), we can rewrite the above sequences as the following commutative diagram of exact sequences;
\begin{eqnarray*}
\rar   H^{n-i}_{\cS_1}(\cS,\Q)   & \sta{ g }{\rar} H^{n-i}(\cS,\Q)\sta{h}{\rar} & W_{n-i}H^{n-i}(\cSU_C^s(r,\cO_C),\Q)\rar 0 \\
         \downarrow \cup \cL^{i}_{\cS_1}  & \downarrow \cup \cL^{i} & \downarrow \cup (\cL^{i}_s)_W  \\
 \rar H^{n+i}_{\cS_1}(\cS,\Q) & \sta{ g'}{\rar} H^{n+i}(\cS,\Q)\sta{h'}{\rar} & W_{n+i} H^{n+i}(\cSU_C^s(r,\cO_C),\Q)\rar 0.
\end{eqnarray*}

\begin{lemma}\label{supportcoh}
The map $\cup (\cL^{i}_s)_W$ between the bottom weight cohomologies in the above long exact sequences is an isomorphism, whenever $n-i$ is odd, as a morphism  of Hodge structures. 
\end{lemma}
\begin{proof}
Since the line bundle $\cL$ restricts to an ample class on $ \cSU_C^s(2,\cO_C)$, the Lefschetz operator $\cup\cL^{i}$ induces an operator on the cohomology $H^{n-i}_{\cS_1}(\cS,\Q)$. Applying the Hard Lefschetz isomorphism on $\cS$, we have the isomorphism of the operator $\cup\cL^{i}$, as morphisms of Hodge structures.
By Lemma \ref{support}, we deduce that the morphisms $h$ and $h'$ are injective and onto the bottom weight cohomologies.
Since $\cup\cL^{i}$ is an isomorphism we obtain that $\cup(\cL^{i}_s)_W$ is also an isomorphism.
\end{proof}

\begin{lemma}\label{HLs}
The Lefschetz operator 
$$
\cup\cL_s^{n-3}:H^3(\cSU_C^s(2,\cO_C),\Q)\rar H^{2n-3}(\cSU_C^s(2,\cO_C),\Q)
$$
is injective as a morphism of Hodge structures. 
\end{lemma}
\begin{proof} 
Consider the above long exact sequences when $n-i =3, n+i=2n-3$.
Using Lemma \ref{supportcoh}, we know that the maps $h,h'$ in the above long exact cohomology sequence on the cohomology are injective.
Since the Hard Lefschetz isomorphism holds on $\cS$, the operator $\cup \cL^{n-3}$ is an isomorphism. Furthermore, by \cite[Theorem 8.3.1]{Ar-Sa}, the cohomology $H^3(\cSU_C^s(2,\cO_C),\Q)$ has a pure Hodge structure and hence the map $h$ is an isomorphism onto $H^3(\cSU_C^s(2,\cO_C),\Q)$. This implies that $\cup \cL_s^{n-3}$ is injective. Moreover, the maps in the above long exact sequence are morphisms of mixed Hodge structures and since $\cup \cL^{n-3}$ is an isomorphism of Hodge structure, we obtain the last assertion.

\end{proof}

\begin{corollary}\label{isoJac}
There is an injectivity
$$
H^1(C,\Q)\hookrightarrow H^{2n-3}(\cSU_C^s(2,\cO_C),\Q)
$$
of Hodge structures.
In particular the intermediate Jacobian 
$$
IJ_1(W_{2n-3})\otimes \Q:= \frac{W_{2n-3}H^{2n-3}(\cSU_C^s(2,\cO_C,\comx)}{F^{n-1}+W_{2n-3}H^{2n-3}(\cSU_C^s(2,\cO_C),\Q)}
$$
is isomorphic to the Jacobian $Jac(C)\otimes \Q$.
\end{corollary}
\begin{proof}
Use \eqref{ar-sa} together with Lemma \ref{HLs} to conclude that $H^1(C,\Q)$ is a sub-Hodge structure of the mixed Hodge structure on $H^{2n-3}(\cSU_C^s(2,\cO_C),\Q)$. Using Lemma \ref{supportcoh}, we deduce the last assertion.
\end{proof}

It was communicated to us by D. Arapura \cite{Ar} that it may be possible to apply the Hard Lefschetz isomorphism
on the intersection cohomology of the moduli space $\cSU_C(r,\cO_C)$, at least in low degrees, to get the desired isomorphism.
Indeed, we check this in the following discussion.

\begin{corollary}\label{lefhighr}
Suppose $\cL$ is an ample class on the moduli space $\cSU_C(r,\cO_C)$, for $r\geq 3$ and $g\geq 4$. Then there is a Hard Lefschetz isomorphism
$$
H^3(\cSU_C^s(r,\cO_C),\Q) \sta{\cup\cL^{n-3}}{\simeq} H^{2n-3}_c(\cSU_C^s(r,\cO_C),\Q).
$$
In particular, we have an isomorphism of the intermediate Jacobians:
$$
Jac(C)\otimes \Q \simeq IJ(H^3(\cSU_C^s(r,\cO_C),\Q)) \simeq IJ(H^{2n-3}_c(\cSU_C^s(r,\cO_C),\Q)).
$$
\end{corollary}
\begin{proof}
By \cite{BBDe}, we have an isomorphism of the intersection cohomology groups of $\cSU_C(r,\cO_C)$:
$$
IH^3(\cSU_C(r,\cO_C)) \sta{\cup\cL^{n-3}}{\simeq} IH^{2n-3}(\cSU_C(r,\cO_C))
$$
given by the Lefschetz operator.
Recall the codimension estimate \cite[Remark 5.1.1]{Ar-Sa}:
$$
\m{codim}(\cSU_C(r,\cO_C)-\cSU_C^s(r,\cO_C))\,>\,5
$$
whenever $r\geq 3$ and $g\geq 4$.
Then, by \cite[p.991, Proposition 3]{Du}, we have the isomorphisms
$$
IH^3(\cSU_C(r,\cO_C)) \simeq H^3(\cSU_C^s(r,\cO_C),\Q) 
$$
and
$$
IH^{2n-3}(\cSU_C(r,\cO_C))\simeq H^{2n-3}_c(\cSU_C^s(r,\cO_C),\Q).
$$
This gives the Hard Lefschetz isomorphism in the degree $3$ rational cohomology group of $\cSU_C(r,\cO_C)$. 
The last claim follows because all the above isomorphisms are isomorphisms of $\Q$-mixed Hodge structures and using the identification in \eqref{ar-sa}.
\end{proof}

 In the following subsection,  we will  determine the codimension two rational Chow group $\CH^2(\cSU_C(2,\cO_C);\Q)_{\hom}$.  
 
\subsection{Isomorphism of the Abel--Jacobi map for codimension two cycles, $r\geq 2$}\label{extmap}

Recall the Hecke diagram from the previous subsection:
\begin{eqnarray*}
\p\,\,\,\,\,\,\,\,\, & \sta{f}{\lrar} & \cSU_C(r,\cO_C) \\
\downarrow \pi \,\,\,\,&     & \\
\cSU_C(r,\cO(-x)).
\end{eqnarray*}

Denote the intermediate Jacobians of a smooth projective variety  $Y$ of dimension $l$ by
$$
IJ^2(Y):= \frac{H^{3}(Y,\comx)}{F^2+ H^3(Y,\Z)},\,\, IJ_1(Y):=\frac{H^{2l-3}(Y,\comx)}{F^l+H^{2l-3}(Y,\Z)}.
$$

We will use these notations for the spaces $\cSU_C(r,\cO(-x))$ and the $\p^1$-bundle $\p$.

\begin{lemma}\label{ijiso}
a) We have the following isomorphisms of the intermediate Jacobians
\begin{eqnarray*}
IJ_1(\p) & \simeq & IJ_1( \cSU_C(r,\cO(-x) ) ) \\ 
IJ^2(\cSU_C(r,\cO(-x))) & \simeq & Jac(C)\otimes \Q.
\end{eqnarray*}
b)
Similar isomorphisms hold for the Chow groups:
\begin{eqnarray*}
\CH_1(\p)_{\hom} & \simeq & \alpha.\pi^*\CH_1(\cSU_C(r,\cO(-x)))_{\hom}, \\
 \CH^2(\p)_{\hom} & \simeq & \CH^2(\cSU_C(r,\cO(-x)))_{\hom},\\
\CH^2(U;\Q)_{\hom} & \sta{\simeq}{\leftarrow} & \CH^2(\cSU_C^s(r,\cO_C);\Q)_{\hom}.
\end{eqnarray*}
Here $\alpha:=c_1(\cO_{\p}(1))$.
\end{lemma}
\begin{proof} a)
There is a decomposition
$$
H^{2n-1}(\p,\Q)=\pi^*H^{2n-1}(\cSU_C(r,\cO(-x)),\Q)\oplus \alpha.\pi^*H^{2n-3}(\cSU_C(r,\cO(-x)),\Q).
$$
The first cohomology piece is zero since it is dual to $H^1(\cSU_C(2,\cO(-x)),\Q)$ which is zero.
This implies that there is an isomorphism of the intermediate Jacobians
$$
 IJ_1(\cSU_C(r,\cO(-x))\simeq IJ_1(\p).
$$
Again, by Corollary \ref{isoIJ}, we have the isomorphism
$$
IJ^2(\cSU_C(r,\cO(-x)))\simeq Jac(C)\otimes \Q.
$$

b)
Using the projective bundle formula \cite[p.64]{Fulton}, similar statements hold on the Chow groups of cycles homologous to zero, as asserted.
Indeed, we note that
\begin{eqnarray*}
\CH^2(\p;\Q) & = & CH_{n-1}(\p; \Q) \\
                   & = & \pi^*CH_{n-2}( \cSU_C(r,\cO(-x));\Q) + \alpha. \pi^*CH_{n-1}(\cSU_C(r,\cO(-x));\Q).\\
\end{eqnarray*}
Since  $\CH_{n-1}(\cSU_C(r,\cO(-x));\Q) =\Q$ (\cite[Proposition 3.4]{Ra}), we conclude the equality
$$
\CH^2(\p;\Q)_{\hom} = CH^{2}( \cSU_C(r,\cO(-x));\Q)_{\hom}.
$$
Also, since $\CH_0(\cSU_C(r,\cO(-x));\Q)=\Q$, the projective bundle formula gives the equality
$$
\CH_1(\p;\Q)_{\hom} =\alpha.\pi^*\CH_1(\cSU_C(r,\cO(-x));\Q)_{\hom}.
$$

Similarly, since $\m{Pic}(\cSU_C^s(r,\cO_C))\otimes \Q=\Q$, we conclude the isomorphism 
$$
\CH^2(U;\Q)_{\hom}\sta{\simeq}{\leftarrow} \CH^2(\cSU_C^s(r,\cO_C);\Q)_{\hom}.
$$
\end{proof}

\begin{proposition}\label{isoPHS}
There is a commutative diagram
\begin{eqnarray*}
\CH^2(\cSU_C(r,\cO(-x));\Q)_{\hom} & \sta{\simeq}{\lrar} & \CH^2(\cSU_C^s(r,\cO_C);\Q)_{\hom} \\
\downarrow AJ^2  &   & \downarrow AJ^2_s \\
Jac(C)\otimes \Q \,\,\,\,& \sta{\simeq}{\lrar}  & Jac(C)\otimes \Q.
\end{eqnarray*}
Furthermore, the maps $AJ^2$ and $AJ^2_s$ are isomorphisms, whenever $g\geq 3$.
\end{proposition}
\begin{proof}
Firstly, by Lemma \ref{ijiso}, we have an isomorphism  $IJ^2(\cSU_C(r,\cO(-x)))\simeq Jac(C)\otimes \Q$ and by
 \eqref{ar-sa} there is an isomorphism  $Jac(C)\otimes \Q\simeq  IJ^2(\cSU_C^s(r,\cO_C))\otimes \Q$.
Here the maps $AJ^2$ is the usual Abel-Jacobi map  and $AJ^2_s$ is the Abel-Jacobi map explained in \S \ref{generalAJ}.
Consider the following maps on the cohomology
\begin{equation}\label{cohom}
H^3(\cSU_C(r,\cO(-x)),\Q)\sta{\pi^*}{\rar}H^{3}(\p,\Q) \sta{s^0}{\rar} H^{3}(U,\Q)\sta{f^*_U}{\leftarrow}H^{3}(\cSU_C^s(r,\cO_C),\Q).
\end{equation}
Then all the above maps are isomorphisms, by \cite[Theorem 8.3.1]{Ar-Sa}.
These are also isomorphisms of pure Hodge structures and hence induce isomorphisms of the associated intermediate Jacobians
\begin{equation}\label{IJP}
IJ^2(\cSU_C(r,\cO(-x)))\simeq IJ^2(\p) \simeq IJ^2(U) \simeq IJ^2(\cSU_C^s(r,\cO_C)).
\end{equation} 

Similarly there are maps between the rational Chow groups
{\Small
\begin{equation}\label{chow}
\CH^2(\cSU_C(r,\cO(-x));\Q)_{\hom} \sta{\pi^*}{\rar} \CH^{2}(\p;\Q)_{\hom}  \sta{t^0}{\rar}    
\CH^2(U;\Q)_{\hom}\sta{f^*_U}{\leftarrow} \CH^2(\cSU_C^s(r,\cO_C);\Q)_{\hom}.
\end{equation}
}
Using Lemma \ref{ijiso}, we conclude that $\pi^*$ and $f^*_U$ are isomorphisms.
Also, there is a localization exact sequence (see \cite[Proposition 1.8, p.21]{Fulton})
$$
 \rar \CH_{n-1}(\p-U; \Q) \rar \CH_{n-1}(\p;\Q) \sta{t^0}{\rar} \CH_{n-1}(U; \Q) \rar 0.
$$
Using the codimension estimate \eqref{codim}, we conclude that $CH_{n-1}(\p-U)\otimes \Q=0$, i.e., $t^0$ is an isomorphism, whenever $g\geq 3$. This implies that all the maps in \eqref{chow} are isomorphisms and compatible with the Abel--Jacobi maps into the objects in \eqref{IJP}. This gives the commutative diagram as in the statement of the lemma. When $r=2$, by  \cite[p.10]{BalajiKingNew}, we know that $AJ^2$ is an isomorphism. This implies that $AJ^2_s$ is also an isomorphism. When $r\geq 2$, we recall the result of Bloch-Srinivas \cite[Theorem 1]{BlochSrinivas}, to obtain an isomorphism of the Abel-Jacobi map 
$$
\CH^2(\cSU_C(2,\cO(-x); \Q)  \simeq Jac(C)\otimes \Q.
$$ 
This is because the moduli space is a smooth and projective  rationally connected variety and hence \cite{BlochSrinivas} is applicable. This suffices to obtain the isomorphism
of $AJ^2_s$.

\end{proof}


\subsection{Some isomorphisms, when $r\geq 2$}

As shown in the $r=2$ case (see Lemma \ref{isoIJ}), we have the following isomorphisms in the higher rank $r\geq 2$ and $g\geq 4$ case as well:

\begin{lemma}\label{isoIJ2}
a) We have the following isomorphisms of the intermediate Jacobians
\begin{eqnarray*}
1) &IJ(H^{2n-1}(\p,\Q))  \,\,\,\,\,\,\,\,\,\,\simeq & IJ(H^{2n-3}(\cSU_C(r,\cO(-x)),\Q) ),\\
2) &IJ(H^{2n-1}_c(U,\Q)) \,\,\,\,\,\,\,\,\,\, \simeq & IJ( H^{2n-3}_c(\cSU_C^s(r,\cO_C),\Q)),\\
3)&IJ(H^{2n-3}(\cSU_C(r,\cO(-x)),\Q) ) \simeq & IJ( H^{2n-3}_c(\cSU_C^s(r,\cO_C),\Q))\\
4)& Jac(C)\otimes \Q \,\,\,\,\,\,\,\,\,\,\,\,\,\,\,\,\,\,\,\,\,\simeq &IJ( H^{2n-3}_c(\cSU_C^s(r,\cO_C),\Q)).
\end{eqnarray*}
b)
Similar isomorphisms hold for the Chow groups:
\begin{eqnarray*}
\CH_1(\p)_{\hom} & \simeq & \alpha.\pi^*\CH_1(\cSU_C(r,\cO(-x)))_{\hom}, \\
 \CH_1(U)_{\hom} & \simeq & \beta.f_U^* \CH_1(\cSU_C^s(r,\cO_C))_{\hom}.
\end{eqnarray*}
Here $\alpha:=c_1(\cO_{\p}(1))$ and $\beta:=c_1(\cO_{U/\cSU_C^s(r,\cO_C)}(1))$.
\end{lemma}
\begin{proof} a)
The proof of 1) is the same as in Lemma \ref{isoIJ} a).
Next for 2), we have a commutative diagram (note that $U\rar \cSU_C^s(r,\cO_C)$ is a $\p^{r-1}$-bundle):
\begin{eqnarray*}
H^3(U,\Q) & \sta{f_U^*}{\simeq} & H^3(\cSU_C^s(r,\cO_C),\Q) \\
\downarrow\simeq & & \downarrow \simeq \\
H^{2n-1}_c(U,\Q) & \sta{\beta.f_U^*}{\leftarrow}  & H^{2n-3}_c(\cSU_C^s(r,\cO_C),\Q).
\end{eqnarray*}
The vertical isomorphisms are by Hard Lefschetz theorems, see Corollary \ref{lefhighr}. This implies that the map $\beta. f_U^*$ is an isomorphism
of pure Hodge structrures (since $H^3(U,\Q)$ has a pure Hodge structure).
The isomorphism in 3) is obtained from the commutative diagram:
\begin{eqnarray*}
H^3(\p,\Q) & \sta{j'}{\simeq} & H^3(U,\Q) \\
\downarrow\simeq & & \downarrow \simeq \\
H^{2n-1}(\p,\Q) & \sta{j}{\rar}  & H^{2n-1}_c(U,\Q).
\end{eqnarray*}
The vertical isomorphisms are by Hard Lefschetz theorems and the isomorphism of $j'$ is due to the codimension estimate
in \cite[p.17]{Ar-Sa}. This shows that $j$ is an isomorphism of pure Hodge structures. Combining the isomorphisms of 1) and 2), we obtain 3) and 4).

b) The proof is the same as in Lemma \ref{isoIJ} b).                  
\end{proof}


\section{Abel-Jacobi maps for a non-compact smooth variety}\label{apprAJ}


\subsection{General Abel-Jacobi maps}\label{generalAJ}

Suppose $X$ is a smooth quasi projective variety defined over the complex numbers. Let $X\subset \ov{X}$ be a smooth compactification of $X$.
Let MHS denote the category of $\Q$-mixed Hodge structures. 
There is an Abel-Jacobi map:
\begin{eqnarray*}
\CH^m(X;\QQ)_{\hom} & \rar & \Ext^1_{\MHS}(\Q(0),H^{2m-1}(X,\Q(m)) \\
              & = & \Ext^1_{\MHS}(\Q(0),W_0H^{2m-1}(X,\Q(m))
\end{eqnarray*}

We are interested in the Abel-Jacobi map restricted to the image:
\[
\CH^m_{\hom}(X;\QQ)^{\circ} := {\rm Im}\big(\CH^m(\ov{X};\QQ)_{\hom}\rar \CH^m(X;\QQ)_{\hom}\big).
\]
As Lewis pointed out, the conjectured equality
\[
\CH^m_{\hom}(X;\QQ)^{\circ} = \CH^m_{\hom}(X;\QQ),
\]
is a consequence of the Hodge conjecture. Indeed one can
show the following (communicated by Lewis):
\begin{prop} Let $Y\subset X$ be a subvariety, where $X/\CC$ is
smooth projective.  For $m \leq 2$ and $m\geq n-1$ (and more generally for
all $m$ if one assumes the Hodge Conjecture), there is an exact sequence:
\[
\CH^m_{{Y}}({X};\QQ)^{\circ} \xrightarrow{j} \CH_{\hom}^m({X};\QQ)\to  \CH^m_{\hom}(X-Y;\QQ)\to 0,
\]
where
\[
\CH^m_{{Y}}({X};\QQ)^{\circ} := \{\xi\in \CH^m_{{Y}}({X};\QQ)\ |\
j(\xi)\in  \CH_{\hom}^m({X};\QQ)\}.
\]
\end{prop}

\begin{proof} Left to the reader.
\end{proof}

Thus we get a map $\CH^m(X;\QQ)_{\hom}^{\circ}\rar$
$$
 {\rm Im}\big(\Ext^1_{\MHS}(\Q(0),W_{-1}H^{2m-1}(X,\Q(m)) 
\to \Ext^1_{\MHS}(\Q(0),W_0H^{2m-1}(X,\Q(m))\big).
$$
Here the latter term can be identified with
\[
 \frac{\Ext^1_{\MHS}(\Q(0),W_{-1}H^{2m-1}(X,\Q(m))}{\delta \hom_{\MHS}(\Q(0),Gr_0^WH^{2m-1}(X,\Q(m)))},
\]
where $\delta$ is the connecting homomorphism in the long exact sequence associated to
$$
0\rar W_{-1}H^{2m-1}({X},\Q(m))\rar W_0H^{2m-1}(X,\Q(m))\rar Gr_0^WH^{2m-1}(X,\Q(m))\rar 0.
$$
In case, $Gr_0^WH^{2m-1}(X,\Q(m))=0$ then the target of the Abel-Jacobi map is the group 
$\Ext^1_{\MHS}(\Q(0),W_{-1}H^{2m-1}(X,\Q(m))$.
In the setting that we consider in the later sections, the relevant cohomologies have a pure Hodge structure. This follows from the main result of Arapura-Sastry \cite[Theorem 1.0.1]{Ar-Sa} and applying Hard Lefschetz theorem to the intersection cohomology of $\cSU_C(r,\cO_C)$. This shows that the compactly supported cohomology of the stable locus $\cSU_C^s(r,\cO_C)$ in degree $2n-3$ has a pure Hodge structure and is identified with $H^1(C,\Q)$ (up to Tate twists), see Corollary \ref{lefhighr}.

Unfortunately, the desired vanishing of the right term $Gr_0$ in the above exact sequence is not true when the rank $r$ is $2$. This could be checked using Seshadri's desingularisation of the moduli space and by looking at the explicit description of the exceptional locus, and via long exact cohomology sequences which relate the cohomologies of the strata. Hence we do not get the Abel-Jacobi map into the required $\Ext^1$ group.

Thus we are led to consider other models of Chow groups for quasi-projective varieties
 which admit good Abel-Jacobi maps. As explained in 
Appendix \S \ref{appendix}, one has well defined Abel-Jacobi maps on relative Chow groups.
In other words, given a smooth projective variety $\ol{X}$ of dimension $n$ and a normal crossing divisor $\ol{Y}\subset \ol{X}$, one has a group of cycles $\ul{\CH}^*(\ol{X},\ol{Y})$, which admits an Abel-Jacobi map:
$$
\ul{\CH}^m(\ol{X},\ol{Y};\Q)_{\hom}\rar IJ(H^{2m-1}_c(X-Y,\Q)).  
$$

We will explain this map in the next section, in the context of a log resolution due to F. Kirwan \cite{Ki}.

\section{Relative Chow groups for Kirwan's log resolution}\label{SEC002}

\subsection{Kirwan's log resolution}\label{KirwansLR}

Consider the moduli space $X:=\cSU_C(2,\cO_C)$ and whose singular locus is $S$, the Kummer variety of 
$Jac(C)$. Denote $U:=X-S$. 
We consider Kirwan's log resolution  \cite{Ki}:
$$
h: \ov{X}  \rar \cSU_C(2,\cO_C).
$$
Denote the exceptional locus by $\ov{Y}\subset \ov{X}$ and $U:= \ov{X}-\ov{Y}=\cSU_C^s(2,\cO_C)$. Then $\ov{Y}$ is a simple normal crossing divisor and the resolution is obtained as a sequence of blow-ups.

\subsection{Abel-Jacobi maps and obstruction to lifting}

We recall definitions from the Appendix \S \ref{appendix}, for Kirwan's resolution. 

In this situation, we are given that
$H^3(U,\QQ)$ is a pure Hodge structure. Hence the restriction
$H^3(X,\QQ) \to H^3(U,\QQ)$ is surjective. Moreover
$H^{2n-3}_c(U,\QQ)$ is also a pure Hodge structure, and the operation $L$ of
intersection with a hyperplane section of $X$, together with the
interpretation of these cohomologies with intersection cohomology
determines an isomorphism of Hodge structure:
\[
L^{n-3} : H^3(U,\QQ) \xrightarrow{\sim} H^{2n-3}_c(U,\QQ).
\]
 Notice that
\[
H^3(U,\QQ) \simeq \frac{H^3(\ol{X},\QQ)}{H^3_{\ol{Y}}(\ol{X},\QQ)},
\]
\[
H^{2n-3}_c(U,\QQ) \simeq H^{2n-3}(\ol{X},\ol{Y},\QQ).
\]
The definition of the appropriate Chow groups must closely mirror the definition
of the respective cohomologies. So in this case
\[
H^3(U,\QQ) \simeq \frac{H^3(\ol{X},\QQ)}{H^3_{\ol{Y}}(\ol{X},\QQ)} \leftrightarrow
\CH^2(U;\QQ) \simeq \CH^2(\ol{X};\QQ)/\CH^2_{\ol{Y}}(\ol{X};\QQ),
\]
where $\CH^2_{\ol{Y}}(\ol{X};\QQ) = \CH^1(\ol{Y};\Q)$ is identified with
its image in $\CH^2(\ol{X};\Q)$.
\[
H^{2n-3}_c(U,\QQ) \simeq H^{2n-3}(\ol{X},\ol{Y},\QQ) \leftrightarrow
\ul{\CH}^{n-1}(\ol{X},\ol{Y};\QQ),
\]
where the latter is relative Chow  cohomology $\ul{\CH}^{\bullet}(\ol{X},\ol{Y})$ (see Appendix \S \ref{appendix}, where it is also pointed out that there is a map
$\ul{\CH}^{\bullet}(\ol{X},\ol{Y}) \to {\CH}^{\bullet}(\ol{X})$). The Abel-Jacobi map:
\[
\CH_{\hom}^{n-1}(\ol{X},\ol{Y};\QQ) \to J\big(H^{2n-3}_c(U,\QQ)\big)
\]
is well-defined, as given in \cite{K-L}.  As pointed out earlier, the AJ map
\[
\CH^2_{\hom}(U,\QQ) \to J\big(H^3(U,\QQ)\big),
\]
is defined by purity of the Hodge structure $H^3(U,\QQ)$. 

Putting the above facts together, we have a diagram:
\[
\begin{matrix}
\CH^2_{\hom}(\ol{X},\ol{Y};\QQ)&\xrightarrow{L^{n-3}}&\CH_{\hom}^{n-1}(\ol{X},\ol{Y};\QQ)\\
\beta\biggl\downarrow\quad&&\biggr|\\
\CH_{\hom}^2(U;\QQ)&&\biggr|\\
\biggl\downarrow&&\biggr\downarrow\\
J\big(H^3(U,\QQ)\big)&\xrightarrow{L^{n-3}}&J\big(H_c^{2n-3}(U,\QQ)\big)
\end{matrix}
\]
The problem here is that we do not know if $\beta$ is surjective. Hence it requires us to understand the Chow group $CH_1(\ov{X};\Q)$ and relate it with the relative Chow group. This is done in \S \ref{identification}, after we understand the geometry of Kirwan's resolution and the Chow generation of one-cycles on this resolution.


\section{Surjectivity of the Abel-Jacobi map, for the usual Chow group of one-cycles}\label{compactAJsurj}


We first note that there is a factorisation of Kirwan's log resolution (see \cite{KiemLi} or \cite{YHK}):
\begin{equation}\label{resKS}
\ov{X}\rar \tilde{S} \rar \cSU_C(2,\cO_C).
\end{equation}
Here $\tilde{S}$ is Seshadri's desingularisation discussed in \S \ref{Seshadri's}. Furthermore, the morphism $\ov{X}\rar \tilde{S}$ is a sequence of two blow-ups.
It was shown by Balaji \cite[Theorem]{Ba2}, that $IJ(H^3(\tilde{S},\Q))\simeq Jac(C)\otimes \Q$.

Since $\ov{X}\rar \tilde{S}$ is a sequence of blow-ups along  smooth rationally connected varieties $R$ (see \cite[p.509, Proposition 5.1]{KiemLi}), using blow-up formula for cohomology, it follows that
$H^3(\ov{X},\Q)\simeq H^3(\tilde{S},\Q)$ (use the fact that $H^1(R,\Q)=0$).
This implies that 
\begin{equation}\label{Kiso}
IJ(H^3(\ov{X},\Q))\simeq Jac(C)\otimes \Q.
\end{equation}

We now prove the result we need in the next section:
\begin{proposition}
The usual Abel-Jacobi map
$$
AJ_1:\CH_1(\ov{X};\Q) \rar IJ(H^{2n-3}(\ov{X};\Q))\simeq Jac(C)\otimes \Q
$$
is surjective.
\end{proposition}
\begin{proof}
Using Hard Lefschetz and \eqref{Kiso}, we first note that
$$
IJ(H^{2n-3}(\ov{X};\Q))\simeq Jac(C)\otimes \Q.
$$
Since the maps in \eqref{resKS} are birational morphisms and $\cSU_C(2,\cO_C)$ is rationally chain connected (see \S \ref{last}, for more details), it follows that $\tilde{S}$ and $\ov{X}$ are also rationally chain connected. 
Hence we can use the result \cite[Theorem 1]{BlochSrinivas}, to obtain an isomorphism
$$
\CH^2(\ov{X};\Q)\simeq Jac(C)\otimes \Q.
$$
Using the compatibility of the Lefschetz operator, we have a commutative diagram:
\begin{eqnarray*}
\CH^2(\ov{X};\Q) & \sta{L^{n-3}}{\rar} & \CH_1(\ov{X};\Q) \\
\downarrow \simeq &                   & \downarrow AJ_1 \\
Jac(C)\otimes \Q & \sta{L^{n-3}}{\rar} & Jac(C)\otimes \Q
\end{eqnarray*}
The Lefschetz operator on $Jac(C)\otimes \Q$ is an isomorphism.
This gives the surjectivity of the Abel-Jacobi map $AJ_1$.

\end{proof}


\section{Chow generation for one cycles on $\cSU_C(r,\cO_C)$}\label{last}

A study of Fano manifolds with Picard number one with respect to the geometry of the variety of tangent directions
to the minimal rational curves, has been studied by J-M. Hwang and N. Mok in a series of papers (see \cite{Hw-Mo} for a survey). The moduli space $\cSU_C^s(2,\cO_C)$ is a Fano manifold with Picard number one and the minimal rational curves are the `Hecke curves' introduced by Ramanan and Narasimhan \cite{Na-Ra2}. Suppose $\cL$ is the ample generator of $\cSU_C(2,\cO_C)$ then the dualizing class $K$ is equal to $-4\cL$ (\cite{Beauville}). Furthermore, a Hecke curve has degree $4$ with respect to $-K$. Hence it has degree one with respect to $\cL$ (see also \cite{Hw2}).
In this section, we include the results of J-M. Hwang, which more generally show that the Hecke curves generate the Chow group of one cycles on the moduli space $\cSU_C(r,\cO_C)$.
 
With notations as in the previous section or \cite{Ar-Sa}, there is a fibration
\begin{eqnarray*}
\cP\,\, & \sta{f}{\rar} \cSU_C(r,\cO_C)\\
\downarrow\!\!\pi && \\
\cS\,\,\, & & \\
\downarrow \!\!\psi && \\
C \,\,\,\\
\end{eqnarray*}
such that the fibre of the morphism $\psi$ at a point $x\in C$ is the moduli space $\cSU_C(r,\cO(-x))$. The variety
$\cP$ is a $\p^{r-1}$-bundle over $\cS$ and restricting over a fibre $\psi^{-1}(x)$ gives precisely the Hecke correspondence used in the previous section. The image under $f$ of the lines in the fibres of the projection $\pi$ are the Hecke curves on $\cSU_C(r,\cO_C)$.

Firstly, we look at a variant of a theorem of Kollar \cite[Proposition 3.13.3]{Ko}, on the Chow generation of one cycles on a variety:

\begin{proposition}\label{generator}(Hwang) 
Let $X$ be a normal projective variety.
Let $\cM \subset {\rm Chow}(X)$ be a closed subscheme of the Chow
scheme of $X$ such that all  members of $\cM$ are irreducible and
reduced curves on $X$. For a general point $x \in X$, let $\cC_x
\subset \p T_x(X)$ be the closure of the union of the tangent
vectors to members of $\cM$ passing through $x$, which are smooth
at $x$. Suppose that $\cC_x$ is non-degenerate (i.e. not contained in any hyperplane) in $\p T_x(X)$.
Then the Chow group of 1-cycles ${\rm CH}_1(X;{\Q})$ is generated
by members of $\cM$. 
\end{proposition}

\begin{proof} Applying Theorem IV.4.13 of \cite{Ko}, we have an
open subvariety $X^o \subset X$ and a morphism $\pi: X^o \to Z^o$
with connected fibers such that for every $z \in Z^o$, any two
general points of  $\pi^{-1}(z)$ can be connected by a connected
chain consisting of   members of $\cM$ of length at most $\dim X$
and all members of $\cM$ through a  point $x \in \pi^{-1}(z)$ are
contained in the closure of $\pi^{-1}(z)$. By the assumption on
the non-degeneracy of $\cC_x$ in $\p T_x(X)$, this implies that
$Z^o$ is a point. It follows that any two general points of $X$
can be connected by a chain consisting of members of $\cM$ of length at
most $\dim X$. Let $Y \subset {\rm Chow}(X)$ be the closed
subscheme parametrizing connected 1-cycles each component of which
is a member of $\cM$. In the notation of Lemma IV.3.4 of
\cite{Ko}, $u^{(2)}: U \times_Y U \to X \times X$ is dominant,
where $g: U \to Y$ is the universal family and $u:U \to X$ is the
cycle map. Since $U$ is complete, $u^{(2)}$ is surjective as in
the proof of Corollary IV.3.5 of \cite{Ko}. Thus we can apply
Proposition 3.13.3 of \cite{Ko} to conclude that ${\rm
CH}_1(X;{\Q})$ is generated by members of $\cM$.
\end{proof}

This proposition can now be used to obtain the Chow generation for one cycles on the moduli space $\cSU_C(r,\cO_C)$.

\begin{corollary}\label{Hecke}(Hwang)
Let $X = \cSU_C(r, \cO_C)$ for a curve $C$ of genus $\geq 4$.
 Then ${\rm CH}_1(X;\Q)$ is generated by
Hecke curves. Moreover, there is a surjection 
$$
\CH_0(C;\Q) \lrar \CH_1(\cSU_C(r,\cO_C); \Q).
$$
\end{corollary}

\begin{proof} By Theorem 3 of \cite{Hw}, the union of the tangent
vectors to Hecke curves through a general point $y \in X$ is a
non-degenerate subvariety $\cC_y\subset \p T_y(X)$. Since Hecke
curves have degree one with respect to the generator of ${\rm
Pic}(X)$, they form a closed subscheme $\cM$ in ${\rm Chow}(X)$.
Thus Proposition \ref{generator} applies. This implies that there is a surjection
$$
\CH_0(\cS; \Q) \rar \CH_1(\cSU_C(r,\cO_C); \Q).
$$
This map is obtained as follows: first consider the case when $r=2$.
There is a map $\pi^*:\CH^l(\cS) \rar CH^l(\cP)$, where $l:=\m{dim }\cS$, which is the same as the map $\pi^*:\CH_0(\cS)\rar \CH_1(\cP)$ (since $\pi:\cP\rar \cS$ is a $\p^1$-bundle). Now consider the composed map
$$
\CH_0(\cS;\Q) \sta{\pi^*}{\rar} \CH_1(\cP;\Q) \sta{f_*}{\rar} \CH_1(\cSU_C(r,\cO_C); \Q).
$$
When $r>2$, we consider the Grassmanian bundle $\rho:\cG(1,r)\rar \cS$ of lines in $\cP$ and $\pi':\cP'\rar \cG(1,r)$ be the universal line. Denote the projection
$f': \cP'\rar \cSU_C(r,\cO_C)$. Since any fibre of
$\rho$ is a rationally connected variety, it follows that $\CH_0(\cG(1,r);\Q)\simeq \CH_0(\cS;\Q)$. Consider the composed map
$$
\CH_0(\cS;\Q)\simeq \CH_0(\cG(1,r); \Q)  \sta{{\pi'}^*}{\rar} \CH_1(\cP';\Q) \sta{f_*}{\rar} \CH_1(\cSU_C(r,\cO_C); \Q).
$$

We have now shown the surjectivity of this map.

For any $x\in C$, the moduli space $\cSU_C(r,\cO(-x))$ is rationally connected and we have the triviality
$\CH_0(\cSU_C(r,\cO(-x));\Q)\,=\,\Q$. This gives us an isomorphism
$$
\CH_0(C;\Q) \simeq CH_0(\cS,\Q).
$$ 
Altogether, we now get a surjection
$$
\CH_0(C;\Q) \lrar CH_1(\cSU_C(r,\cO_C); \Q).
$$

\end{proof}

\section{Chow generation for one cycles on Kirwan's log resolution and the Abel-Jacobi isomorphism, when $r=2$}


We now restrict our attention to the case when the rank $r=2$. 

Recall Kirwan's desingularisation, from \S \ref{KirwansLR}:
$$
h: \ov{X}\rar \cSU_C(2,\cO_C)
$$
such that $\ov{Y}\subset \ov{X}$ is a normal crossing divisor in the smooth projective variety $\ov{X}$.

From the previous subsection, we know that the Hecke curves are the minimal degree rational curves on the moduli space $\cSU_C(2,\cO_C)$. Since the morphism $h$ is a birational morphism, Hecke curves pass through a general point of the resolution $\ov{X}$ as well. Let $\cH\subset \mbox{Chow}(\ov{X})$ be the closed subscheme of the Chow variety of $\ov{X}$ whose general member is the Hecke curve in $\ov{X}$. Let $\cP'\rar \cH$ be the universal proper family of rational curves, whose general fibre is irreducible.

Notice that, given a point $x\in C$, there is a Hecke correspondence
\begin{eqnarray*}
\cP_x\,\, & \sta{f}{\rar} \cSU_C(r,\cO_C)\\
\downarrow\!\!\pi && \\
\cSU_C(2,\cO(-x))\,\,\, & & 
\end{eqnarray*}
such that the images of the lines in the $\p^1$-bundle $\cP_x$ are precisely the Hecke curves. This implies that
the closed subscheme $\cH$ can be written as 
$$
\cH\,=\,\cup_{x\in C}\cH_x.
$$
Here $\cH_x\subset \mbox{Chow}(\ov{X})$ is the closed subscheme whose general member is a Hecke curve which is the image of lines of the $\p^1$-bundle $\cP_x$.

This gives a birational map:
$$
\zeta_x: \cH_x \dasharrow  \cSU_C(2,\cO(-x)).
$$

Consider a blow-up of this rational map, to get a proper birational morphism:
$$
\tilde{\eta_x}:\tilde{\cH_x}\rar \cSU_C(2,\cO(-x)).
$$
Since $\cSU_C(2,\cO(-x))$ is a rationally connected variety, the variety $\cH_x$ and hence $\tilde{\cH_x}$  are rationally chain connected varieties, for each $x\in C$. Consider the family $\tilde{\cH}=\bigcup_{x\in C}\tilde{\cH_x}\rar C$. Then we obtain an isomorphism of rational Chow group of $0$-cycles:
\begin{equation}\label{chowiso}
\CH_0(C;\Q) \simeq \CH_0(\tilde\cH;\Q).
\end{equation}

Consider the universal proper family of  rational curves $\tilde\eta:\tilde{\cP}:=\cP\times_{\cS}\tilde{\cH}\rar \tilde{\cH}$ which is the pullback of the $\p^1$ bundle via the above morphism $\zeta:=\cup_{x\in C}\zeta_x:\tilde{\cH}\rar \cS$.  Note that due to properness, the family $\tilde{\cP}\rar \tilde{\cH}$ maps surjectively onto the resolution $\ov{X}$. Indeed the projection onto $\ov{X}$, factors via the family $\cP'\rar \cH$ and there is a commutative diagram:
\begin{eqnarray*}
\tilde{\cP} & \rar \cP' &\rar \ov{X}\\
\downarrow \tilde{\eta} &  \,\, \downarrow & \\
\tilde{\cH} & \rar \cH. & \\
\end{eqnarray*}

 In other words, the map
$$
f':\tilde{\cP}\rar \ov{X}
$$
is a proper and surjective morphism. Furthermore, all the fibres of $\tilde{\cP}\rar \tilde{H}$ are irreducible.

\begin{lemma}\label{genKirwan}
The (irreducible) members of the family $\tilde{\cP}\rar \tilde{\cH}$ generate the Chow group $\CH_1(\ov{X}; \Q)$.
\end{lemma}
\begin{proof}
Since the map $h:\ov{X}\rar \cSU_C(2,\cO_C)$ is a birational morphism,  $\ov{X}$ is also a rationally connected smooth proper variety. Furthermore, any two general points are connected by a chain of Hecke curves (since the same is true on $\cSU_C(2,\cO_C)$). By properness, any two points of $\ov{X}$ are connected by a chain of rational curves, which are parametiesd by $\tilde{\cH}$. Also note that $\tilde{\cP}\rar \tilde{\cH}$ is a family of (irreducible) rational curves and we note  that the projection $\tilde{\cP}\times_{\tilde{\cH}}\tilde{\cP} \rar \ov{X} \times \ov{X}$ is surjective. 

Applying Proposition \ref{generator} or Theorem IV.4.13 of \cite{Ko}, we deduce that the (irreducible) members of the covering family $\tilde{\cP}\rar \tilde{\cH}$ generate the Chow group $\CH_1(\ov{X}, \Q)$.

\end{proof}

\begin{corollary}\label{surjSigma}
There is a surjection
$$
\CH_0(C,\Q)\rar \CH_1(\ov{X};\Q).
$$
\end{corollary}
\begin{proof}
Using Lemma \ref{genKirwan} we know that the irreducible fibres of $\tilde{\cP}\rar \tilde{\cH}$ generate $\CH_1(\ov{X};\Q)$.  Together with \eqref{chowiso}, we deduce that the composed map 
$$
\CH_0(C;\Q)\simeq \CH_0(\tilde{\cH};\Q)\sta{\tilde{\eta}^*}{\rar} \CH_1(\tilde{\cP};\Q)\sta{f_*}{\rar}\CH_1(\ov{X};\Q)
$$
is a surjection.
\end{proof}  

From \S \ref{compactAJsurj}, we have a surjective map
\begin{equation}\label{AJSigma}
\CH^{n-1}_{\hom}(\ol{X};\QQ)
\rightarrow
J\big(H^{2n-3}(\ov{X},\QQ)\big) \simeq Jac(C)\otimes \Q.
\end{equation}

\begin{corollary}\label{motive}
The surjective Abel-Jacobi map 
$$
AJ_1:\CH_1(\ov{X};\Q)_{\hom} \lrar Jac(C)\otimes \Q
$$
is an isomorphism, whenever  $g\geq 4$ .
\end{corollary}
\begin{proof}
Using \eqref{AJSigma} and Corollary \ref{surjSigma}, we note that all the maps in the following sequence
\begin{equation}\label{arrow}
Jac(C)\otimes \Q \rar \CH_1(\ov{X};\Q)_{\hom}  \rar Jac(C)\otimes \Q
\end{equation}
are surjective. Furthermore, all the maps are in fact motivated by correspondences. So a multiple of the composed map $Jac(C)\otimes \Q\rar Jac(C)\otimes \Q$ lifts to a surjective endomorphism of abelian varieties, and hence it is an isomorphism. Hence all the maps in \eqref{arrow} are isomorphisms.
\end{proof}

\subsection{ Identification of the relative Chow group with the Chow group of one cycles}\label{identification}

Recall that there is a sequence (see \eqref{E27}):
\begin{equation}
 \ul{\CH}^r(\ol{Y},1) \to \ul{\CH}^r(\ol{X},\ol{Y})\to \CH^r(\ol{X})\xrightarrow{\alpha} \ul{\CH}^r(\ol{Y}).
\end{equation}
Our interest here is the case $r=n-1$. One has the Gysin map
$$
Gy^{\ast}:=\alpha:\CH^{n-1}(\ol{X};\QQ) \to  \ul{\CH}^{n-1}(\ol{Y};\QQ).
$$
We refer to the pair $(\ov{X},\ov{Y})$, where $\ov{Y}$ is the NCD
in Kirwan's log resolution $\ol{X}$.

\begin{lemma}\label{Gysinzero}
The Gysin map $Gy^*$ is the zero map.
\end{lemma}
\begin{proof}
Recall the Hecke correspondence, for any $x\in C$:
\begin{eqnarray*}
{\cP_x} & \rar & \cSU_C(2,\cO_C) \\
\downarrow   & & \\
\cSU_C(2,\cO(-x)).
\end{eqnarray*}
The images of the fibers of the $\p^1$-bundle are the Hecke curves on $\cSU_C(2,\cO_C)$.
Suppose that $\alpha_0$ is a one-cycle on $\cSU_C(2,\cO_C)$, and $\alpha_0$ intersects the singular locus  non trivially. Since $CH_1(\cSU_C(2,\cO_C);\Q)$ is generated by Hecke curves 
(see Corollary \ref{Hecke}), we can assume that $\alpha_0$ is a Hecke curve.
Now most of the Hecke curves lie completely, on the open part $\cSU_C^s(2,\cO_C)$ (see the codimension estimates in \cite[\S 4]{Ar-Sa}, for $g\geq 3$). Since the parameter space $\cSU_C(2,\cO(-x))$ is a rationally connected variety, the cycle $\alpha_0$ can be moved in its rational equivalence to a cycle which is a Hecke curve lying completely inside $\cSU_C^s(2,\cO_C)$. 

This argument can be repeated for the following correspondence, and for any $x\in C$:
\begin{eqnarray*}
{\tilde{\cP_x}} & \rar & \ov{X} \\
\downarrow   & & \\
\tilde{\cH_x}.
\end{eqnarray*}

We again observe that the fibers of this $\p^1$-bundle generate $CH_1(\ov{X};\Q)$, when we take all $x\in C$ (see Lemma \ref{genKirwan}). Due to the above codimension estimates, most of these curves lie completely inside $\ov{X}-\ov{Y}=\cSU_C^s(2,\cO_C)$. Since $\tilde{\cH_x}\rar \cSU_C(2,\cO_C(-x))$ is a birational morphism, $\tilde{\cH_x}$ is a rationally chain connected variety. Hence if $\alpha_0$ is any curve from this family, then $\alpha_0$ is rationally equivalent to a cycle lying completely inside $\ov{X}-\ov{Y}$. This implies that $Gy^*(\alpha_0)$ is zero.

\end{proof} 

\begin{corollary}\label{gysinzerocor}
We have an isomorphism
$$
\frac{\ul{\CH}^{n-1}(\ol{X},\ol{Y};\Q)}{\ul{\CH}^{n-1}(\ol{Y},1)} \simeq \CH_1(\ov{X};\Q).
$$
\end{corollary}

\subsection{Proof of main theorem}

We use the identification, proved in Corollary \ref{gysinzerocor}:

$$
\CH^{n-1}(\ol{X};\QQ) = 
 {\rm Im}\big(\ul{\CH}^{n-1}(\ol{X},\ol{Y};\QQ)\to \CH^{n-1}(\ol{X};\QQ)\big)
$$
 \[
\simeq \frac{\ul{\CH}^{n-1}(\ol{X},\ol{Y};\QQ)}{\ul{\CH}^{n-1}(\ol{Y},1;\QQ)}.
\]
and note that the below Abel-Jacobi maps commute
\begin{eqnarray*}
 \ul{\CH}^{n-1}_{\hom}(\ov{X},\ov{Y};\Q) & \rar & \CH^{n-1}_{\hom}(\ov{X};\Q) \\
 \downarrow AJ^c & & \downarrow AJ_1 \\
 IJ(H^{2n-3}_c(U;\Q)) & \rar & Jac(C)\otimes \Q.
 \end{eqnarray*}

Hence, we now deduce our main Theorem on the induced AJ map (see \eqref{inducedAJ});

\begin{theorem}
Given a smooth connected projective curve $C$ of genus $g\geq 4$, consider  Kirwan's log resolution $h:\ov{X}\rar \cSU_C(2,\cO_C)$. The inverse image of the singular locus is a normal crossing divisor $\ov{Y}\subset \ov{X}$. Then the induced Abel-Jacobi map
\[
 \frac{\ul{\CH}^{n-1}_{\hom}(\ol{X},\ol{Y};\QQ)}{\ul{\CH}^{n-1}(\ol{Y},1;\QQ)^o}\lrar IJ (H^{2n-3}_c(U,\Q))\simeq Jac(C)\otimes \Q
\]
is an isomorphism.
\end{theorem}

\section{Appendix: Relative Chow groups and Abel-Jacobi maps, by J. Lewis}\label{appendix}

\subsection{Relative Chow groups}\label{cohomCH}

Most of the ideas for this section originate from \cite{K-L}; however
it is much easier for the reader to see these ideas translated
in the context of this paper.
Suppose $X/\CC$ is a projective variety of dimension $n$,
and $Y\subset X$ a proper closed algebraic subset containing
the singular locus of $X$. Put $U = X\bs Y$.
Consider a proper modification $(\ol{X},\ol{Y})$
of the pair $(X,Y)$, where $\ol{X}$ is now smooth projective, 
$\ol{Y}\subset \ol{X}$ is a NCD, and $U = \ol{X}\bs \ol{Y}$.
The definition of relative Chow cohomology
suitable for our purposes can be found in \cite{K-L}, which also coincides
with the corresponding description of the MHS  $H^{i}(\ol{X},\ol{Y},\Q(j))$.
Write the NCD
\[
\ol{Y} = \bigcup_{i=0}^N\ol{Y}_i,
\]
where $\ol{Y}_i$ is a smooth divisor in $\ol{X}$. For an integer
$t\geq 0$, put $\ol{Y}^{[t]} =$  disjoint union of the $t$-fold intersections
of the  components of $\ol{Y}$ with corresponding coskeleton
$\ol{Y}^{[\bullet]}$. We also put $\ol{Y}^{(t)}$ to be the union of the
$t$-fold intersections of the components of $\ol{Y}$, where we observe
that $Y^{[t]}$ is the ``canonical'' desingularization
of $\ol{Y}^{(t)}$. We have a simplicial
complex
\[
\ol{Y}^{[\bullet]} \xrightarrow{\Gy} \ol{Y}\subset \ol{X} =: \ol{Y}^{[0]},
\]
which is used to calculate $\ul{\CH}^r(\ol{X},\ol{Y},m)$, where $m\geq 0$
includes Bloch's higher Chow groups \cite{Bloch} as well. The definition even in
the case $m=0$ requires some formalism related to the higher Chow
groups. Recall the $m$-simplex
\[
\Delta^m := \Spec\biggl(\frac{\CC[t_0,...,t_m]}{\big(1-\sum_{j=0}^mt_j\big)}\biggr),
\]
with $j$-th face given by $t_j=0$. Then put
\[
z^r(\ol{X},m) := \big\{\xi\in z^r(\ol{X}\times \Delta^m)\ \big|\ \xi\ {\rm meets \ all \  multiindex\
faces\ properly}\big\}.
\]
So one has $\del := \sum_{j=0}^m(-1)^j\del_j : z^r(\ol{X},m) \to z^r(\ol{X},m-1)$,
with $\del^2=0$. We will also assume that $z^r(\ol{X},m)$ is
defined so that all such $\xi$ meet $\ol{Y}^{(\bullet)}$
properly as well. We then have a double complex with associated simple
complex and differential $D := \del \pm \Gy^{\ast}$:
\begin{equation}\label{E66}
\begin{matrix}
&\del\biggl\downarrow\quad&&\del\biggl\downarrow\quad&&
\del\biggl\downarrow\quad\\
\cdots\xleftarrow{\Gy^{\ast}}&z^r(\ol{Y}^{[2]},2)&\xleftarrow{\Gy^{\ast}}&z^r
(\ol{Y}^{[1]},2)&\xleftarrow{\Gy^{\ast}}&z^r(\ol{X},2)\\
&\del\biggl\downarrow\quad&&\del\biggl\downarrow\quad&&
\del\biggl\downarrow\quad\\
\cdots\xleftarrow{\Gy^{\ast}}&z^r(\ol{Y}^{[2]},1)&\xleftarrow{\Gy^{\ast}}&z^r
(\ol{Y}^{[1]},1)&\xleftarrow{\Gy^{\ast}}&z^r(\ol{X},1)\\
&\del\biggl\downarrow\quad&&\del\biggl\downarrow\quad&&
\del\biggl\downarrow\quad\\
\cdots\xleftarrow{\Gy^{\ast}}&z^r(\ol{Y}^{[2]},0)&\xleftarrow{\Gy^{\ast}}&z^r
(\ol{Y}^{[1]},0)&\xleftarrow{\Gy^{\ast}}&z^r(\ol{X},0)
\end{matrix}
\end{equation}
The cohomology of the diagonal 
\[
\bigoplus_{i+j=0; i\leq 0,j\geq 0}z^r(\ol{Y}^{[-i]},j),
\]
computes the relative Chow cohomology $\ul{\CH}^r(\ol{X},\ol{Y})$ introduced
in \cite{K-L}. [Note: Eliminating the right hand column of
(\ref{E66}) amounts to a double complex that computes
the Chow cohomology $\ul{\CH}^{\bullet}(\ol{Y},\bullet)$. Note that
for smooth projective $V$, $\ul{\CH}^r(V) = \CH^r(V)$, where the latter is
the usual Fulton Chow (homology) group.] Note that there is a natural map
$$
\ul{\CH}^r(\ol{X},\ol{Y}) \to \CH^r(\ol{X})
$$ 
which mirrors the
cohomological situation 
$$H^{\bullet}(\ol{X},\ol{Y}) \to H^{\bullet}(\ol{X}).
$$
Formally from (\ref{E66}) we have an exact sequence of \underbar{cohomology} groups:
\begin{equation}\label{E27}
\to  \ul{\CH}^r(\ol{Y},1) \to \ul{\CH}^r(\ol{X},\ol{Y})\to \CH^r(\ol{X})\xrightarrow{\alpha} \ul{\CH}^r(\ol{Y}),
\end{equation}
where, as in the cohomological situation,
 $\alpha$ need not be surjective (e.g. $\ol{X} = \PP^3$, $\ol{Y}\subset \ol{X}$ a smooth surface of degree $\geq 4$ and $r=2$). Note that there is a map
\[
\ul{\CH}^r(\ol{Y}) \to \CH^r(\ol{Y}^{[1]}),
\]
and that under this map:
\[
\alpha\big(\CH^r(\ol{X})\big) \ {\rm goes\ to}\   \CH_{\ker\Gy}^r(\ol{Y}^{[1]}).
\]
Now let us return to the situation of Kirwan's
resolution $(\ol{X},\ol{Y})$. Regarding the map $\ul{\CH}^{n-1}(\ol{Y},1) \to \CH^{n-1}(\ol{X},\ol{Y})$, let
$\ul{\CH}^{n-1}(\ol{Y},1)^{\circ}$ be the inverse image of $\CH_{\hom}^{n-1}(\ol{X},\ol{Y})$.
By definition of the Chow cohomology groups, there are maps:
\[
\ul{\CH}^{n-1}(\ol{X},\ol{Y};\Q) \to H^{2n-2}_{\Dd}(\ol{X},\ol{Y},\Q(n-1)),
\]
\[
\ul{\CH}_{\hom}^{n-1}(\ol{X},\ol{Y};\Q) \to J\big(H^{2n-3}(\ol{X},\ol{Y},\Q(n-1))\big)
\]
\[
\ul{\CH}^{n-1}(\ol{Y},1;\Q) \to H^{2n-2}_{\Dd}(\ol{Y},\Q(n-1)).
\]
This makes use of the  short exact sequences:
$$
0\to J\big(H^{2n-4}(\ol{X},\ol{Y},\Q(n-1))\big) \to H^{2n-2}_{\Dd}(\ol{X},\ol{Y},\Q(n-1))
\to  \Gamma \big(H^{2n-3}(\ol{X},\ol{Y},\Q(n-1))\big)\to 0,
$$
$$
0\to J\big(H^{2n-4}(\ol{Y},\Q(n-1))\big) \to H^{2n-2}_{\Dd}(\ol{Y},\Q(n-1))
\to  \Gamma \big(H^{2n-3}(\ol{Y},\Q(n-1))\big)\to 0;
$$
moreover by a weight argument, 
\[
\Gamma\big(H^{2n-3}(\ol{Y},\QQ(n-1))\big) := \hom_{\MHS}\big(\QQ(0),H^{2n-3}(\ol{Y},\QQ(n-1))\big)
= 0.
\]
But by purity of the Hodge structure $H^{2n-3}(\ol{X},\ol{Y},\QQ(n-1)) \simeq H^{2n-3}_c(U,\QQ(n-1))$,
where $U = \ol{X}-\ol{Y}$, the image 
\[
J\big(H^{2n-4}(\ol{Y},\QQ(n-1))\big) \to  J\big(H^{2n-3}(\ol{X},\ol{Y},\QQ(n-1))\big),
\]
is zero. Hence there is an induced Abel-Jacobi map:
\begin{equation}\label{inducedAJ}
\frac{\CH^{n-1}_{\hom}(\ol{X},\ol{Y};\QQ)}{\ul{\CH}^{n-1}(\ol{Y},1;\QQ)^{\circ}}
\to J\big(H^{2n-3}(\ol{X},\ol{Y},\QQ(n-1))\big).
\end{equation}


\end{document}